\documentclass[11pt]{amsart}

\usepackage[all]{xy}
\usepackage{amsmath,amssymb,amsthm,enumerate}
\usepackage[colorlinks, bookmarks=true]{hyperref}

\theoremstyle{plain}
\newtheorem{theorem}{Theorem}[section]
\newtheorem{corollary}[theorem]{Corollary}
\newtheorem{lemma}[theorem]{Lemma}
\newtheorem{proposition}[theorem]{Proposition}

\theoremstyle{definition}

\newtheorem{remark}[theorem]{Remark}
\newtheorem{question}[theorem]{Question}

\newtheorem{case}{Case}

\newtheorem{conjecture}[theorem]{Conjecture}

\renewcommand{\leq}{\leqslant}
\renewcommand{\geq}{\geqslant}\usepackage{amssymb}

\newcommand{\vr}{\varepsilon}
\newcommand{\one}{\mathbf{1}}

\newcommand{\be}{\begin{equation}}
\newcommand{\ee}{\end{equation}}

\newcommand{\N}{\mathbb N}

\newcommand{\ave}{{\mathrm{Ave}}}

\newcommand{\ce}{\mathcal E}
\newcommand{\se}{{{\mathfrak{C}}}_{\mathcal{E}}}

\newcommand{\ball}{\mathbf{B}}
\newcommand{\A}{\mathcal{A}}

\newcommand{\cs}{\mathfrak{C}}
\newcommand{\dist}{\mathrm{dist}}
\newcommand{\lore}{\mathfrak{l}}
\newcommand{\injtens}{\check{\otimes}}
\newcommand{\projtens}{\hat{\otimes}}
\newcommand{\anytens}{\tilde{\otimes}}

\def \rank{{\mathrm{rank}} \, }
\def \ker{{\mathrm{ker}} \, }

\renewcommand{\span}{\mathrm{span}}
\def \vr{\varepsilon}

\def \ran{{\mathrm{ran}} \, }

\def \diag{{\mathrm{diag}} \, }

\def \eqalign#1{\null\,\vcenter{\openup\jot 
   \ialign{\strut\hfil$\displaystyle{##}$&$
      \displaystyle{{}##}$\hfil \crcr#1\crcr}}\,}
\def\iK{\mathcal {K}}

\def\iSS{\mathcal {SS}}

\def\query#1{\setlength\marginparwidth{80pt}%
\marginpar{\raggedright\fontsize{10}{10}\selectfont\itshape
\hrule\smallskip
{\textcolor{red}{#1}}\par\smallskip\hrule}}

\begin{document}

\numberwithin{equation}{section}

\pagestyle{headings}


\title[Subprojective Banach spaces]{Subprojective Banach spaces}

\author[T. Oikhberg]{Timur Oikhberg}
\address{
Dept.~of Mathematics, University of Illinois at Urbana-Champaign, Urbana IL 61801, USA}
\email{oikhberg@illinois.edu}

\author[E. Spinu]{Eugeniu Spinu}
\address{Dept. of Mathematical and Statistical Sciences, University of Alberta
Edmonton, Alberta  T6G 2G1, CANADA}
\email{espinu@ualberta.ca}


\begin{abstract}
A Banach space $X$ is called subprojective if any of its infinite dimensional
subspaces $Y$ contains a further infinite dimensional subspace complemented in $X$.
This paper is devoted to systematic study of subprojectivity.
We examine the stability of subprojectivity of Banach spaces under various operations,
such us direct or twisted sums, tensor products, and forming spaces of operators.
Along the way, we obtain new classes of subprojective spaces.
\end{abstract}

\thanks{The authors acknowledge the generous support of Simons Foundation,
via its Travel Grant 210060. They would also like to thank the organizers of
Workshop in Linear Analysis at Texas A\&M, where part of this work was
carried out. Last but not least, they express their gratitude to
L. Bunce, T. Schlumprecht, B. Sari, and N.-Ch. Wong for many helpful suggestions.}

\maketitle

\section{Introduction and main results}\label{s:intro}

We examine various aspects of subprojectivity.
Throughout this note, all Banach spaces are assumed to be infinite dimensional,
and subspaces, infinite dimensional and closed, until specified otherwise.


A Banach space $X$ is called \emph{subprojective} if every subspace $Y \subset X$
contains a further subspace $Z \subset Y$, complemented in $X$.
This notion was introduced in \cite{Whi}, in order to study the (pre)adjoints
of strictly singular operators. Recall that an operator $T \in B(X,Y)$ is
\emph{strictly singular} ($T \in \iSS(X,Y)$) if $T$ is not an isomorphism
on any subspace of $X$. In particular, it was shown that, if $Y$ is subprojective,
and, for $T \in B(X,Y)$, $T^* \in \iSS(Y^*,X^*)$, then $T \in \iSS(X,Y)$.

Later, connections between subprojectivity and perturbation classes were discovered.
More specifically, denote by $\Phi_+(X,Y)$ the set of \emph{upper semi-Fredholm operators} --
that is, operators with closed range, and finite dimensional kernel. If
$\Phi_+(X,Y) \neq \emptyset$, we define the \emph{perturbation class}
$$
P \Phi_+ (X, Y ) = \{ S \in B(X,Y) : T+S \in \Phi_+(X,Y)
{\textrm{   whenever  }} T \in \Phi_+(X,Y) \}. 
$$
It is known that $\iSS(X, Y ) \subset P \Phi_+ (X, Y )$. In general, this
inclusion is proper. However, we get $\iSS(X, Y ) = P \Phi_+ (X, Y )$ if $Y$
is subprojective  (see \cite[Theorem 7.51]{Ai} for this, and for similar
connections to inessential operators).


Several classes of subprojective spaces are described in \cite{GMASB}.
Common examples of non-subprojective
space are $L_1(0,1)$ (since all  Hilbertian subspaces of $L_1$ are not complemented), 
$C(\Delta)$, where $\Delta$ is the Cantor set, or $\ell_\infty$
(for the same reason). The disc algebra is not subprojective, since by e.g.
\cite[III.E.3]{Woj} it contains a copy of $C(\Delta)$.
By \cite{Whi}, $L_p(0,1)$ is subprojective if and only if $2 \leq p < \infty$.
Consequently, the Hardy space $H_p$ on the disc is subprojective for
exactly the same values of $p$. Indeed, $H_\infty$ contains the disc algebra.
For $1 < p < \infty$, $H_p$ is isomorphic to $L_p$. The space $H_1$ contains
isomorphic copies of $L_p$ for $1 < p \leq 2$ \cite[Section 3]{WojH1}.
On the other hand, VMO is subprojective (\cite{MuS}, see also \cite{RanVMO}
for non-commutative generalizations).


We start our paper by collecting various facts needed to study subprojectivity
(Section \ref{s:stability}). Along the way, we prove that subprojectivity is stable
under suitable direct sums (Proposition \ref{p:dir_sum}). However, subprojectivity
is not a $3$-space property (Proposition \ref{p:3 space}). Consequently, subprojectivity
is not stable under the gap metric (Proposition \ref{p:SP gap}). Considering the place
of subprojective spaces in Gowers dichotomy, we observe that each subprojective
space has a subspace with an unconditional basis. However, we exhibit a space
with an unconditional basis, but with no subprojective subspaces (Proposition \ref{p:GM}).

In Section \ref{s:tens_prod}, we investigate the
subprojectivity of tensor products, and of spaces of operators.
A general result on tensor products (Theorem \ref{t:ce est})
yields the subprojectivity on $\ell_p \injtens \ell_q$ and $\ell_p \projtens \ell_q$
for $1 \leq p, q < \infty$ (Corollary \ref{c:l_p l_q}), as well as of $\iK(L_p, L_q)$ for
$1 < p \leq 2 \leq q < \infty$ (Corollary \ref{c:p est}).
We also prove that the space $B(X)$ is never subprojective (Theorem \ref{t:B(X)}),
and give an example of non-subprojective tensor product $\ell_2 \otimes_\alpha \ell_2$
(Proposition \ref{p:tens not SP}).

Throughout Section \ref{s:cont}, we work with $C(K)$ spaces, with $K$ compact metrizable.
We begin by observing that $C(K)$ is subprojective if and only if $K$ is scattered. 
Then we prove that $C(K,X)$ is subprojective if and only if both $C(K)$ and $X$ are
(Theorem \ref{t:C(K,X)}). Turning to spaces of operators, we show that,
for $K$ scattered, $\Pi_{qp}(C(K),\ell_q) $ is subprojective (Proposition \ref{p:P pq}).
Then we study continuous fields on a scattered base space, proving that any scattered
separable CCR $C^*$-algebra is subprojective (Corollary \ref{c:field}).

Section \ref{s:schatten} shows that, in many cases, subprojectivity passes from
a sequence space to the associated Schatten spaces (Proposition \ref{p:sch_subpr}).

Proceeding to Banach lattices, 
in Section \ref{s:DH} we prove that  $p$-disjointly homogeneous $p$-convex
lattices ($2 \leq p < \infty$) are subprojective (Proposition \ref{p:p-DH l2 lp}).
In Section \ref{s:lattice} (Proposition \ref{p:X(l_p)}), we show that
the lattice $\widetilde{X(\ell_p)}$ is subprojective whenever $X$ is.
Consequently (Proposition \ref{p:Rad X}), if $X$ is a subprojective space
with an unconditional basis and non-trivial cotype, then $Rad(X)$ is subprojective.



Throughout the paper, we use the standard Banach space results and notation.
By $B(X,Y)$ and $\iK(X,Y)$ we denote the sets of linear bounded and compact operators,
respectively, acting between Banach spaces $X$ and $Y$. $\ball(X)$ refers to the
closed unit ball of $X$. For $p \in [1,\infty]$, we denote by $p^\prime$ the ``adjoint''
of $p$ (that is, $1/p + 1/p^\prime = 1$).


\section{General facts about subprojectivity}\label{s:stability}

We begin this section by showing that subprojectivity passes to direct sums.

\begin{proposition}\label{p:dir_sum}
(a) Suppose $X$ and $Y$ are Banach spaces. Then the following are equivalent:
\begin{enumerate}
\item
Both $X$ and $Y$ are subprojective.
\item
$X \oplus Y$ is subprojective.
\end{enumerate}

(b) Suppose $X_1, X_2, \ldots$ are Banach spaces, and $\ce$ is a space
with a $1$-unconditional basis. Then the following are equivalent:
\begin{enumerate}
\item
The spaces $\ce, X_1, X_2, \ldots$ are subprojective.
\item
$(\sum_n X_n)_\ce$ is subprojective.
\end{enumerate}
\end{proposition}

In (b), we view $\ce$ as a space of sequences of scalars, equipped
with the norm $\| \cdot \|_\ce$. $(\sum_n X_n)_\ce$ refers to
the space of all sequences $(x_n)_{n \in \N} \in \prod_{n \in \N} X_n$,
endowed with the norm $\|(x_n)_{n \in \N}\| = \|(\|x_n\|_{X_n})\|_\ce$.
Due to the $1$-unconditionality (actually, $1$-suppression unconditionality
suffices), $(\sum_n X_n)_\ce$  is a Banach space.

We begin by making two simple observations, to be used several times throughout this paper.

\begin{proposition}\label{p:compare}
Consider Banach spaces $X$ and $X^\prime$, and $T \in B(X,X^\prime)$.
Suppose $Y$ is a subspace of $X$, $T|_Y$ is an isomorphism, and $T(Y)$
is complemented in $X^\prime$. Then $Y$ is complemented in $X$.
\end{proposition}

\begin{proof}
If $Q$ is a projection from $X^\prime$ to $T(Y)$, then
$T^{-1} Q T$ is a projection from $X$ onto $Y$.
\end{proof}

This immediately yields:

\begin{corollary}\label{c:complem}
Suppose $X$ and $X^\prime$ are Banach spaces, and $X^\prime$ is subprojective.
Suppose, furthermore, that $Y$ is a subspace of $X$, and there exists
$T \in B(X,X^\prime)$ so that $T|_Y$ is an isomorphism. Then $Y$ contains
a subspace complemented in $X$.
\end{corollary}

The following version of ``Principle of Small Perturbations'' is folklore,
and essentially contained in \cite{BP58}.
We include the proof for the sake of completeness.

\begin{proposition}\label{p:seq}
Suppose $(x_k)$ is a seminormalized basic sequence in a
Banach space $X$, and $(y_k)$ is a sequence so that
$\lim_k \|x_k - y_k\| = 0$. Suppose, furthermore, that
every subspace of $\span[y_k : k \in \N]$ contains a subspace
complemented in $X$. Then $\span[x_k : k \in \N]$ contains a subspace
complemented in $X$.
\end{proposition}

\begin{proof}
Replacing $x_k$ by $x_k/\|x_k\|$, we can assume that $(x_k)$ normalized.
Denote the biorthogonal functionals by $x_k^*$, and
set $K = \sup_k \|x_k^*\|$. Passing to a subsequence, we can assume
that $\sum_k \|x_k - y_k\| < 1/(2K)$. Define the operator
$U \in B(X)$ by setting $U x = \sum_k x_k^* (x) (y_k - x_k)$.
Clearly $\|U\| < 1/2$, and therefore, $V = I_X + U$ is invertible.
Furthermore, $V x_k = y_k$. If $Q$ is a projection from
$X$ onto a subspace $W \subset \span[y_k : k \in \N]$, then 
$P = V^{-1} Q V$ is a projection from
$X$ onto a subspace $Z \subset \span[x_k : k \in \N]$.
\end{proof}

\begin{remark}\label{r:kernels}
Note that, in the proof above, the kernels and the ranges of the projections
$Q$ and $P$ are isomorphic, via the action of $V$.
\end{remark}



\begin{proof}[Proof of Proposition \ref{p:dir_sum}]
It is easy to see that subprojectivity is inherited by subspaces.
Thus, in both (a) and (b), only the implication $(2) \Rightarrow (1)$
needs to be established.

(a) 
Throughout the proof, $P_X$ and $P_Y$ stand for the coordinate projections from
$X \oplus Y$ onto $X$ and $Y$, respectively.
We have to show that any subspace $E$ of $X \oplus Y$ contains a further
subspace $G$, complemented in $X \oplus Y$.

Show first that $E$ contains a subspace $F$ so that either $P_X|_F$
or $P_Y|_F$ is an isomorphism. Indeed, suppose $P_X|_F$ is not an
isomorphism, for any such $F$. Then $P_X|_E$ is strictly singular,
hence there exists a subspace $F \subset E$, so that $P_X|_F$ has
norm less than $1/2$. But $P_X + P_Y = I_{X \oplus Y}$, hence,
by the triangle inequality, $\|P_Y f\| \geq \|f\| - \|P_X f\| \geq
\|f\|/2$ for any $f \in F$. Consequently, $P_Y|_F$ is an isomorphism.

Thus, by passing to a subspace, and relabeling if necessary, we can assume
that $E$ contains a subspace $F$, so that $P_X|_F$ is an isomorphism.
By Corollary \ref{c:complem}, $F$ contains a subspace $G$, complemented in $X$.

Set $F^\prime = P_X(F)$, and
let $V$ be the inverse of $P_X : F \to F^\prime$. By the subprojectivity
of $X$, $F^\prime$ contains a subspace $G^\prime$, complemented in $X$ via
a projection $Q$.
Then $P = V Q P_X$ gives a projection onto $G = V(G^\prime) \subset F$.

(b) Here, we denote by $P_n$ the coordinate projection from $X = (\sum_k X_k)_\ce$
onto $X_n$. Furthermore, we set $Q_n = \sum_{k=1}^n P_k$, and
$Q_n^\perp = \one - Q_n$. We have to show that any subspace $Y \subset X$
contains a subspace $Y_0$, complemented in $X$. To this end, consider
two cases.

(i) For some $n$, and some subspace $Z \subset Y$, $Q_n|_Z$ is an isomorphism.
By part (a), $X_1 \oplus \ldots \oplus X_n = Q_n(X)$ is subprojective.
Apply Corollary \ref{c:complem} to obtain $Y_0$. 

(ii) For every $n$, $Q_n|_Y$ is not an isomorphism -- that is, for every $n \in \N$,
and every $\vr > 0$, there exists a norm one $y \in Y$ so that
$\|Q_n y\| < \vr$. Therefore, for every sequence of positive numbers
$(\vr_i)$, we can find $0 = N_0 < N_1 < N_2 < \ldots$, and
a sequence of norm one vectors $y_i \in Y$,
so that, for every $i$, $\|Q_{N_i} y_i\|, \|Q^\perp_{N_{i+1}} y_i\| < \vr_i$.
By a small perturbation principle, we can assume that $Y$ contains norm one
vectors $(y_i^\prime)$ so that $Q_{N_i} y_i^\prime = Q^\perp_{N_{i+1}} y_i^\prime = 0$
for every $i$.
Write $y_i^\prime = (z_j)_{j=N_i+1}^{N_{i+1}}$, with $z_j \in X_j$.
Then $Z = \span[(0, \ldots, 0, z_j, 0, \ldots) : j \in \N]$
($z_j$ is in $j$-th position) is complemented in $X$.
Indeed, if $z_j \neq 0$, find $z_j^* \in X_j^*$ so that $\|z_j^*\| = \|z_j\|^{-1}$,
and $\langle z_j^*, z_j \rangle = 1$. If $z_j = 0$, set
$z_j^* = 0$. For $x = (x_j)_{j \in \N} \in X$, define
$R x = (\langle z_j^*, x_j \rangle z_j)_{j \in \N}$.
It is easy to see that $R$ is a projection onto $Z$,
and $\|R\|$ does not exceed the unconditionality constant of $\ce$.

Now note that
$J : Z \to \ce : (\alpha_1 z_1, \alpha_2 z_2, \ldots) \mapsto
 (\alpha_1 \|z_1\|, \alpha_2 \|z_2\|, \ldots)$
is an isometry. Let
$Y^\prime = \span[y_i^\prime : i \in \N]$, and $Y_\ce = J(Y^\prime)$. By the
subprojectivity of $\ce$, $Y_\ce$ contains a subspace $W$, which is complemented
in $\ce$ via a projection $R_1$. Then $J^{-1} R_1 J R$ is a projection from $X$
onto $Y_0 = J^{-1}(W) \subset X$.
\end{proof}

\begin{remark} From the last proposition it follows  the (strong) $p$-sum of
subprojective  Banach spaces is subprojective.
On the other hand, the infinite weak sum of subprojective spaces need not be subprojective.

Recall that if $X$ is a Banach space, then 
$$\ell_p^{weak}(X)=\{ x=(x_n)_{n=1}^{\infty} \in X\times X\times X \ldots:
 \sup_{x^* \in X^{*}}{(\sum{|x^* (x_n)|^p})^{\frac{1}{p}}}<\infty \}.$$

It is known that $\ell_p^{weak}(X)$ is isomorphic to $B(\ell_{p^\prime}, X)$
($\frac{1}{p}+\frac{1}{p^\prime}=1$), see \cite[Theorem~2.2]{DJT}. We show that,
for $X=\ell_r$ $(r \ge p^\prime)$, $B(\ell_{p^\prime}, X)$ contains a copy of $\ell_\infty$,
and therefore, is not subprojective. To this end, denote by $(e_i)$ and $(f_i)$
the canonical bases in $\ell_r$ and $\ell_{p^\prime}$ respectively.
For $\alpha = (\alpha_i) \in \ell_\infty$, define
$B(\ell_{p^\prime}, X) \ni U \alpha : e_i \mapsto \alpha_i f_i$.
Clearly, $U$ is an isomorphism.

Note that the situation is different for $r < p^\prime$. Then, by Pitt's Theorem,
$B(\ell_{p^\prime}, \ell_r) = \iK(\ell_{p^\prime}, \ell_r)$. In the next section
we prove that the latter space is subprojective.
\end{remark}

Next we show that subprojectivity is not a $3$-space property.

\begin{proposition}\label{p:3 space}
For $1 < p < \infty$ there exists a non-subprojective Banach space $Z_p$,
containing a subspace $X_p$, so that $X_p$  and $Z_p/X_p$ are isomorphic to $\ell_p$.
\end{proposition}

\begin{proof}
\cite[Section 6]{KP} gives us a short exact sequence
$$ 0 \longrightarrow \ell_p \overset{j_p}{\longrightarrow} Z_p
 \overset{q_p}{\longrightarrow} \ell_p \longrightarrow 0 , $$
where the injection $j_p$ is strictly cosingular, and the quotient map
$q_p$ is strictly singular. By \cite[Theorem 6.2]{KP}, $Z_p$ is not
isomorphic to $\ell_p$. By \cite[Theorem 6.5]{KP}, any non-strictly singular
operator on $Z_p$ fixes a copy of $Z_p$. Consequently, $j_p(\ell_p)$ contains
no complemented subspaces (by \cite[Theorem 2.a.3]{LT1}, any complemented
subspace of $\ell_p$ is isomorphic to $\ell_p$).
\end{proof}


It is easy to see that subprojectivity is stable under isomorphisms.
However, it is not stable under a rougher measure of ``closeness'' of
Banach spaces -- the gap measure. If $Y$ and $Z$ are subspaces of a
Banach space $X$, we define the \emph{gap} (or \emph{opening})
$$
\Theta_X(Y,Z) = \max\big\{ \sup_{y \in Y, \|y\| = 1} \dist(y,Z),
\sup_{z \in Z, \|z\| = 1} \dist(z,Y) \big\} .
$$
We refer the reader to the comprehensive survey \cite{Os94} for more information.
Here, we note that $\Theta_X$ satisfies a ``weak triangle inequality'',
hence it can be viewed as a measure of closeness of subspaces.
The following shows that subprojectivity is not stable under $\Theta_X$.

\begin{proposition}\label{p:SP gap}
There exists a Banach space $X$ with a subprojective subspace $Y$ so that,
for every $\vr > 0$, $X$ contains a non-subprojective space $Z$ with $\Theta_X(Y,Z) \leq \vr$.
\end{proposition}

\begin{proof}
Our $Y$ will be isomorphic to $\ell_p$, where $p \in (1,\infty)$ is fixed.
By Proposition \ref{p:3 space}, there exists a non-subprojective Banach space $W$,
containing a subspace $W_0$, so that both $W_0$  and $W^\prime = W/W_0$ are isomorphic to $\ell_p$.
Denote the quotient map $W \to W^\prime$ by $q$. Consider $X = W \oplus_1 W^\prime$ and
$Y = W_0 \oplus_1 W^\prime \subset E$. Furthermore, for $\vr > 0$, define
$Z_\vr = \{\vr w \oplus_1 qw : w \in W\}$. Clearly, $Y$ is isomorphic to $\ell_p \oplus \ell_p \sim \ell_p$,
hence subprojective, while $Z_\vr$ is isomorphic to $W$, hence not subprojective.
By \cite[Lemma 5.9]{Os94}, $\Theta_X(Y,Z_\vr) \leq \vr$.
%
\end{proof}



Looking at subprojectivity through the lens of Gowers dichotomy and observing that  a subprojective Banach space does not  contain  hereditarily indecomposable subspaces,  we immediately obtain the following. 

\begin{proposition}\label{p:unc bas}
Every subprojective space has a subspace with an unconditional basis.
\end{proposition}


The converse to the above proposition is false.

\begin{proposition}\label{p:GM}
There exists a Banach space with an unconditional basis, without subprojective
subspaces.
\end{proposition}

\begin{proof}
In \cite[Section 5]{GM97}, T.~Gowers and B.~Maurey construct
a Banach space $X$ with a $1$-unconditional
basis, so that any operator on $X$ is a strictly singular perturbation of a diagonal operator.
We prove that $X$ has no subprojective subspaces. In doing so, we are re-using the notation of
that paper. In particular, for $n \in \N$ and $x \in X$, we define $\|x\|_{(n)}$ as the supremum
of $\sum_{i=1}^n \|x_i\|$, where $x_1, \ldots, x_n$ are successive vectors so that $x = \sum_i x_i$.
It is known that, for every block subspace $Y$ in $X$, every $c > 1$, and every $n \in \N$, there
exists $y \in Y$ so that $1 = \|y\| \leq \|y\|_{(n)} < c$. This technical result can be used
to establish a remarkable property of $X$: suppose $Y$ is a subspace of $X$, with a
normalized block basis $(y_k)$. Then any zero-diagonal (relative to the basis $(y_k)$)
operator on $Y$ is strictly singular. Consequently, any $T \in B(Y)$ can be written as
$T = \Lambda + S$, where $\Lambda$ is diagonal, 
and $S$ is zero-diagonal, hence strictly singular.
This result is proved in \cite{GM97} for $Y = X$, but an inspection yields the
generalization described above.

Suppose, for the sake of contradiction, that $X$ contains a subprojective subspace $Y$.
A small perturbation argument shows we can assume $Y$ to be a block subspace.
Blocking further, we can assume that $Y$ is spanned by a block basis $(y_j)$, so that
$1 = \|y_j\| \leq \|y_j\|_{(j)} < 1 + 2^{-j}$. We achieve the desired contradiction
by showing that no subspace of $Z = \span[y_1+y_2, y_3+y_4, \ldots]$ is complemented in $Y$.

Suppose $P$ is an infinite rank projection from $Y$ onto a subspace of $Z$.
Write $P = \Lambda + S$, where $S$ is a strictly singular operator with zeroes
on the main diagonal, and $\Lambda = (\lambda_j)_{j=1}^\infty$ is a diagonal operator
(that is, $\Lambda y_j = \lambda_j y_j$ for any $j$).
As $\sup_j \|y_j\|_{(j)} < \infty$, by \cite[Section 5]{GM97} we have $\lim_j S y_j = 0$.
Note that $(\Lambda + S)^2 = \Lambda + S$, hence
$\diag (\lambda_j^2 - \lambda_j) = \Lambda^2 - \Lambda = S - \Lambda S - S \Lambda - S^2$
is strictly singular, or equivalently, $\lim_j \lambda_j (1 - \lambda_j) = 0$.
Therefore, there exists a $0-1$ sequence $(\lambda_j^\prime)$
so that $\Lambda^\prime - \Lambda$ is compact (equivalenty,
$\lim_j (\lambda_j - \lambda_j^\prime) = 0$), where
$\Lambda^\prime = \diag(\lambda_j^\prime)$ is a diagonal projection.
Then $P = \Lambda^\prime + S^\prime$, where $S^\prime = S + (\Lambda - \Lambda^\prime)$
is strictly singular, and satisfies $\lim_j S^\prime y_j = 0$.
The projection $P$ is not strictly singular (since it is of infinite rank),
hence $\Lambda^\prime = P - S^\prime$ is not strictly singular. Consequently, the set
$J = \{j \in \N : \lambda_j^\prime = 1\}$ is infinite.

Now note that, for any $j$, $\|P y_j - y_j\| \geq 1/2$.
Indeed, $P y_j \in Z$, hence we can write $P y_j = \sum_k \alpha_k (y_{2k-1} + y_{2k})$.
Let $\ell = \lceil j/2 \rceil$. By the $1$-unconditionality of our basis,
$\|y_j - P y_j\| \geq \|y_j - \alpha_\ell (y_{2\ell-1} + y_{2\ell})\| \geq
\max\{|1-\alpha_\ell|, |\alpha_\ell|\} \geq 1/2$. For $j \in J$,
$S^\prime y_j = P y_j - y_j$, hence $\|S^\prime y_j\| \geq 1/2$, which contradicts
$\lim_j \|S^\prime y_j\| = 0$.
\end{proof}

\begin{remark}
The preceding statement provides an example of an atomic order continuous Banach lattice without subprojective subspaces.
One can also observe that if a Banach lattice is not order continuous, then it contains a subprojective subspace $c_0$. Also, if a Banach lattice is non-atomic order continuous with an unconditional basis, then it contains a subprojective subspace $\ell_2$ (i.e. \cite[Theorem~2.3]{KW}). 
\end{remark}

Finally, one might ask whether, in the definition of subprojectivity, the projections from $X$ onto $Z$
can be uniformly bounded. More precisely, we call a Banach space $X$ \emph{uniformly subprojective}
({\emph{with constant $C$}}) if, for every subspace $Y \subset X$, there exists a subspace $Z \subset Y$
and a projection $P : X \to Z$ with $\|P\| \leq C$. The proof of \cite[Proposition 2.4]{GMASB}
essentially shows that the following spaces are uniformly subprojective:
(i) $\ell_p$ ($1 \leq p < \infty$) and $c_0$;
(ii) the Lorentz sequence spaces $\lore_{p,w}$;
(iii) the Schreier space;
(iv) the Tsirelson space;
(v) the James space.
Additionally, $L_p(0,1)$ is uniformly subprojective for $2 \leq p < \infty$.
This can be proved by combining Kadets-Pelczynski dichotomy with the results of
\cite{Al:09} about the existence of ``nicely complemented'' copies of $\ell_2$.
Moreover, any $c_0$-saturated separable space is uniformly subprojective, since
any isomorphic copy of $c_0$ contains a $\lambda$-isomorphic copy of $c_0$,
for any $\lambda > 1$ \cite[Proposition 2.e.3]{LT1}. By Sobczyk's Theorem, a
$\lambda$-isomorphic copy of $c_0$ is $2\lambda$-complemented in every separable superspace.
In particular, if $K$ is a countable metric space, then $C(K)$ is uniformly subprojective
\cite[Theorem 12.30]{FHHMPZ}.

However, in general, subprojectivity need not be uniform. Indeed, suppose
$2 < p_1 < p_2 < \ldots < \infty$, and $\lim_n p_n = \infty$.
By Proposition \ref{p:dir_sum}(b), $X = (\sum_n L_{p_n}(0,1))_2$
is subprojective. The span of independent Gaussian random variables in $L_p$
(which we denote by $G_p$) is isometric to $\ell_2$. Therefore, by \cite[Corollary 5.7]{GLR},
any projection from $L_p$ onto $G_p$ has norm at least $c_0 \sqrt{p}$, where 
$c_0$ is a universal constant. Thus, $X$ is not uniformly subprojective.

\section{Subprojectivity of tensor products and spaces of operators}\label{s:tens_prod}

Suppose $X_1$, $X_2$, \ldots, $X_k$ are Banach spaces with unconditional FDD, implemented by finite
rank projections $(P^\prime_{1n})$, $(P^\prime_{2n})$, \ldots, $(P^\prime_{kn})$, respectively.
That is, $P^\prime_{in} P^\prime_{im} = 0$ unless $n=m$,
$\lim_N \sum_{n=1}^N P^\prime_{in} = I_{X_i}$ point-norm, and
$\sup_{N,\pm} \|\sum_{n=1}^N \pm P^\prime_{in}\| < \infty$ (this quantity is sometimes
referred to as \emph{the FDD constant of $X_i$}). Let $E_{in} = \ran(P^\prime_{in})$.

We say that a sequence $(w_j)_{j=1}^\infty \subset X_1 \otimes X_2 \otimes  \ldots \otimes X_k$ is
\emph{block-diagonal} if there exists a sequence $0 = N_1 < N_2 < \ldots$
so that 
$$
w_j \in \big(\sum_{n=N_j+1}^{N_{j+1}} E_{1n}\big) \otimes
\big(\sum_{n=N_j+1}^{N_{j+1}} E_{2n}\big)\otimes \ldots  \otimes
\big(\sum_{n=N_j+1}^{N_{j+1}} E_{kn}\big).
$$
Suppose $\ce$ is an unconditional sequence space,
and $\anytens$ is a tensor product of
Banach spaces. The Banach space $X_1  \anytens X_2 \anytens \ldots \anytens X_k$ is said to \emph{satisfy the $\ce$-estimate}
if there exists a constant $C \geq 1$ so that, for any block diagonal sequence
$(w_j)_{j \in \N}$ in $X_1  \anytens X_2 \anytens \ldots \anytens X_k$, we have
\begin{equation}
C^{-1} \|(\|w_j\|)_{j \in \N}\|_\ce \leq \|\sum_j w_j\| \leq C \|(\|w_j\|)_{j \in \N}\|_\ce
\label{eq:ce est}
\end{equation}

\begin{theorem}\label{t:ce est}
Suppose $X_1$, $X_2$, \ldots, $X_k$ are subprojective Banach spaces with unconditional FDD,
and $\anytens$ is a tensor product.
Suppose, furthermore, that for any finite increasing sequence
${\mathbf{i}} = [1 \leq i_1 < \ldots < \ldots i_\ell \leq k]$, there
exists an unconditional sequence space $\ce_{{\mathbf{i}}}$, so that
$X_{i_1}  \anytens X_{i_2} \anytens \ldots \anytens X_{i_k}$ satisfies
the $\ce_{\mathbf{i}}$-estimate. Then
$X_1  \anytens X_2 \anytens \ldots \anytens X_k$ is subprojective.
\end{theorem}

A similar result for ideals of operators holds as well.
We keep the notation for projections implementing the FDD in Banach spaces $X_1$ and $X_2$.
We say that a Banach operator ideal $\A$ is suitable (for the pair $(X_1,X_2)$) if
the finite rank operators are dense in $\A(X_1, X_2)$ (in its ideal norm).
We say that a sequence $(w_j)_{j \in \N} \subset \A(X_1, X_2)$ is \emph{block diagonal} 
if there exists a sequence $0 = N_1 < N_2 < \ldots$ so that, for any $j$,
$w_j = (P_{2,N_j} - P_{2,N_{j-1}}) w_j (P_{1,N_j} - P_{1,N_{j-1}})$.
If $\ce$ is an unconditional sequence space, we say that $\iK(X_1, X_2)$
\emph{satisfies the $\ce$-estimate} if, for some constant $C$,
\begin{equation}
C^{-1} \|(\|w_j\|)_j\|_{\ce} \leq \|\sum_j w_j\|_{\A} \leq C \|(\|w_j\|)_j\|_{\ce}
\label{eq:ce ext ops}
\end{equation}
holds for any finite block-diagonal sequence $(w_j)$.

\begin{theorem}\label{t:ce est ops}
Suppose $X_1$ and $X_2$ are Banach spaces with unconditional FDD, so that
$X_1^*$ and $X_2$ are subprojective. Suppose, furthermore,that the ideal $\A$
is suitable for $(X_1,X_2)$, and $\A(X_1,X_2)$ satisfies the $\ce$-estimate
for some unconditional sequence $\ce$. Then $\A(X_1, X_2)$ is subprojective.
\end{theorem}

Before proving these theorems, we state a few consequences.

\begin{corollary}\label{c:l_p l_q}
The spaces $X_1 \injtens \ldots \injtens X_n$ and
$X_1 \projtens \ldots \projtens X_n$ are subprojective where $X_i$
is ether isomorphic to $\ell_{p_i}$ $(1 \le p_i <\infty)$ or $c_0$
for every $ 1 \le i \le n$.
\end{corollary}

For $n=2$, this result goes back to \cite{Sam} and \cite{Oja}
(the injective and projective cases, respectively).

Suppose a Banach space $X$ has an FDD implemented by projections $(P_n^\prime)$ --
that is, $P_n^\prime P_m^\prime = 0$ unless $n=m$,
$\sup_{N,\pm} \|\sum_{n=1}^N \pm P_n\| < \infty$, and
$\lim_N \sum_{n=1}^N P_n = I_X$ point-norm. We say that $X$ satisfies
\emph{the lower $p$-estimate} if there exists a constant $C$ so that,
for any finite sequence $\xi_j \in \ran P_j$,
$\|\sum_j \xi_j\|^p \geq C \sum_j \|\xi_j\|^p$.
The smallest $C$ for which the above inequality holds is called
\emph{the lower $p$-estimate constant}.
\emph{The upper $p$-estimate}, and \emph{the upper $p$-estimate constant},
are defined in a similar manner.
Note that, if $X$ is an unconditional sequence space, then the above definitions
coincide with the standard one (see e.g. \cite[Definition 1.f.4]{LT2}).

\begin{corollary}\label{c:p est}
Suppose the Banach spaces $X_1$ and $X_2$ have unconditional FDD, satisfy
the lower and upper $p$-estimates respectively, and both $X_1^*$ and $X_2$
are subprojective. Then $\iK(X_1,X_2)$ is subprojective.
\end{corollary}

Before proceeding, we mention several instances where the above corollary
is applicable. Note that, if $X$ has type $2$ (cotype $2$), then $X$ satisfies the upper
(resp. lower) $2$-estimate.
Indeed, suppose $X$ has type $2$, and $w_1, \ldots, w_n$ are such that $w_j = P_j w_j$
for any $j$. Then
$$
\|\sum_j w_j\| \leq C \ave_\pm \|\sum_j \pm w_j\| \leq
C T_2(X) \Big( \sum_j \|w_j\|^2 \Big)^{1/2}
$$
($T_2(X)$ is the type $2$ constant of $X$). The cotype case is handled similarly.
Thus, we can state:

\begin{corollary}\label{c:type}
Suppose the Banach spaces $X_1$ and $X_2$ have unconditional FDD,
cotype $2$ and type $2$ respectively, and both $X_1^*$ and $X_2$
are subprojective. Then $\iK(X_1,X_2)$ is subprojective.
\end{corollary}

This happens, for instance, if $X_1 = L_p(\mu)$ or ${\mathfrak{C}}_p$ ($1 < p \leq 2$)
and $X_2 = L_q(\mu)$ or ${\mathfrak{C}}_q$ ($2 \leq q < \infty$).
Indeed, the type and cotype of these spaces are well known
(see e.g. \cite{PX}). The Haar system provides an unconditional basis for $L_p$.
The existence of unconditional FDD of $\cs_p$ spaces is given by \cite{AL}.

\begin{proof}[Proof of Theorem \ref{t:ce est}]
We will prove the theorem by induction on $k$. Clearly, we can take $k=1$ as
the basic case. Suppose the statement of the theorem holds for a tensor product
of any $k-1$ subprojective Banach spaces
that satisfy $\ce$-estimate. We will show that the statement holds
for the tensor product of $k$ Banach spaces
$X=X_1  \anytens X_2 \anytens \ldots \anytens X_k$.

For notational convenience, let $P_{in} = \sum_{k=1}^n P^\prime_{ik}$,
and $I_i = I_{X_i}$.
If $A \in B(X)$ is a projection, we use the notation $A^\perp$ for $I_X - A$.
Furthermore, define the projections
$Q_n = P_{1n} \otimes P_{2n} \otimes \ldots \otimes P_{kn}$
and $R_n = P_{1n}^\perp \otimes P_{2n}^\perp \ldots \otimes P_{kn}^\perp$.
Renorming all $X_i$'s if necessary, we can assume that their unconditional FDD
constants equal $1$.

First show that, for any $n$, $\ran R_n^\perp$ is subprojective.
To this end, write $R_n^\perp = \sum_{i=1}^k P^{(i)}$,
where the projections $P^{(i)}$ are defined by
$$
\begin{array}{rcl}
P^{(1)}  & = &  P_{1n} \otimes I_2 \otimes \ldots \otimes I_k  ,  \\
P^{(2)}  & = &  P_{1n}^\perp \otimes P_{2n} \otimes I_3 \otimes \ldots \otimes I_k ,  \\
P^{(3)}  & = &  P_{1n}^\perp \otimes P_{2n}^\perp \otimes P_{3n} \otimes
 I_4 \otimes \ldots \otimes I_k  , \\
\ldots  &  \ldots  &   \ldots   \\
P^{(k)}  & = &  P_{1n}^\perp \otimes P_{2n}^\perp \otimes \ldots \otimes
 P_{k-1,n}^\perp \otimes P_{kn}
\end{array}
$$
(note also that $P^{(i)} P^{(j)} = 0$ unless $i = j$). Thus, there exists $i$ so that
$P^{(i)}$ is an isomorphism on a subspace $Y^\prime \subset Y$. Now observe that the range of $P^{(i)}$ is isomorphic to
a subspace of $\ell_\infty^N(X^{(i)})$, where $N = \rank P_{in}$, and
$$
X^{(i)} = X_1 \anytens X_2 \anytens \ldots \anytens X_{i-1} \anytens X_{i+1}
 \anytens \ldots \anytens X_k .
$$
By the induction hypothesis, $X^{(i)}$ is subprojective.
By Proposition \ref{p:dir_sum}, $\ran P^{(i)}$ is subprojective for every $i$, hence
so is $\ran R_n^\perp$.

Now suppose $Y$ is an infinite dimensional subspace of $X$.
We have to show that $Y$ contains a subspace $Z$, complemented in $X$.
If there exists $n \in \N$ so that $R_n^\perp|_Y$ is not strictly singular,
then, by Corollary \ref{c:complem}, $Z$ contains
a subspace complemented in $X$.

Now suppose $R_n^\perp|_Z$ is strictly singular for any $n$. It is easy to see that,
for any sequence of positive numbers $(\vr_m)$, one can find
$0 = n_0 < n_1 < n_2 < \ldots$,
and norm one elements $x_m\in Y$, so that, for any $m$,
$\|R_{n_{m-1}}^\perp x_m\| + \|x_m - Q_{n_m} x_m\| < \vr_m$.
By a small perturbation, we can assume that
$x_m = R_{n_{m-1}}^\perp Q_{n_m}x_m$. That is,
$$
x_m \in \ran \Big( (P_{1,n_m} - P_{1,n_{m-1}})\otimes (P_{2,n_m} - P_{2,n_{m-1}})
 \otimes \ldots \otimes (P_{k,n_m} - P_{k,n_{m-1}}) \Big) .
$$
Let $E_{im} = \ran (P_{i,n_m} - P_{i,n_{m-1}})$, and
$W = \span[E_{1m} \otimes E_{2m}\otimes \ldots \otimes E_{km} :m \in \N] \subset X$.
Applying ``Tong's trick'' (see e.g. \cite[p. 20]{LT1}), and
taking the $1$-unconditionality of our FDDs into account, we see that
$$U : X \to W :
 x \mapsto \sum_m \big((P_{1,n_m} - P_{1,n_{m-1}})\otimes \ldots \otimes (P_{k,n_m} - P_{k,n_{m-1}})\big) x$$
defines a contractive projection onto $W$. Furthermore, $Z = \span[x_m : m \in \N]$
is complemented in $W$. Indeed, the projection $P_{i,n_m} - P_{i,n_{m-1}}$ $( i,m  \in \N)$ is contractive, hence we can identify
$E_{1m} \anytens \ldots \anytens E_{km}$ with $(E_{1m} \otimes \ldots \otimes E_{km}) \cap X$.
By the by Hahn-Banach Theorem, for each $m$ there exists a contractive projection
$U_m$ on $E_{1m} \anytens \ldots \anytens E_{2m}$, with range $\span[x_m]$.
By our assumption, there exists an unconditional sequence space $\ce$ so that
$X_1 \anytens \ldots \anytens X_k$ satisfies the $\ce$-estimate.
Then, for any finite sequence
$w_m \in E_{1m} \anytens\ldots \anytens E_{km}$, \eqref{eq:ce est} yields
$$
\|\sum_k U_k w_k\| \leq C \|(\|U_k w_k\|)\|_\ce \leq C \|(\|w_k\|)\|_\ce \leq
C^2 \|\sum_k U_k w_k\| .
$$
Thus, $Z$ is complemented in $X$.
\end{proof}

\begin{proof}[Sketch of the proof of Theorem \ref{t:ce est ops}]
On $\A(X_1,X_2)$ we define the projection
$R_n : \A(X_1,X_2) \to \A(X_1,X_2) : w \mapsto P_{2n}^\perp w P_{1n}$.
Then the range of $R_n^\perp$ is isomorphic to
$X_1^* \oplus \ldots \oplus X_1^* \oplus X_2 \oplus \ldots \oplus X_2$.
Then proceed as in the the proof of Theorem \ref{t:ce est}
(with $k=2$).
\end{proof}

To prove Corollary \ref{c:l_p l_q}, we need two auxiliary results.

\begin{lemma}\label{l:inj est}
Suppose $1 < p_i < \infty$ $(1\le i \le n)$ and $X=\injtens_{i=1}^n \ell_{p_i}$. 
\begin{enumerate}
\item
If $\sum1/p_i>n-1$, then  $X$  satisfies the $\ell_s$-estimate with $1/s = \sum{1/p_i} - (n-1)$.
\item
 If $\sum 1/p_i \leq n-1$, then $X$ satisfies the $c_0$-estimate.
\end{enumerate}
\end{lemma}

\begin{proof}
Suppose $(w_j)$ is a finite block-diagonal sequence in $X$.
We shall show that $\|\sum_j w_j\| = \|(\|w_j\|)\|_s$,
with $s$ as in the statement of the lemma.
To this end, let $(U_{ij})$
be coordinate projections on $\ell_{p_i}$ for every  $1\le i \le n$, such that
$w_j = U_{1j} \otimes \ldots \otimes U_{nj} w_j$, and for each $i$,  $U_{ik} U_{im} = 0$ unless $k=m$. Letting
$p_i^\prime = p_i/(p_i-1)$, we see that
$$
\|\sum_j w_j\| = \sup_{\xi _i\in \ell_{p_i^\prime}, 
 \| \xi_i\| \leq 1} \Big| \langle \sum_j w_j, \otimes_i \xi_i \rangle \Big| .
$$
Choose $ \otimes_i \xi_i$ with $\|\xi_i\| \le 1$,  and let $\xi_{ij} = U_{ij} \xi_i$. Then
$\sum_j \|\xi_{ij}\|^{p_i^\prime} \leq 1$,
and
$$
\Big| \langle \sum_j w_j, \otimes_i \xi_i \rangle \Big| \leq
\sum_j \big| \langle w_j, \otimes_i \xi_i \rangle \big| =
\sum_j \big| \langle w_j, \otimes_i\xi_{ij} \rangle \big| \leq
\sum_j \|w_j\| \Pi_{i=1}^n \|\xi_{ij}\|.
$$
Now let $1/r = \sum1/p_i^\prime =n - \sum{1/p_i}$. By H\"older's Inequality,
$$
\Big( \sum_j \big( \prod_{i=1}^n  \|\xi_{ij}\|  \big)^r \Big)^{1/r} \leq
\prod_{i=1}^n \Big( \sum_j \|\xi_{ij}\|^{p_i^\prime} \Big)^{1/p_i^\prime} \leq 1 .
$$
If $\sum 1/p_i \leq n-1$, then $r \le 1$, hence $\sum_j \Pi_{i=1}^n\|\xi_{ij}\| \leq 1$.
Therefore, $\|\sum_j w_j\| \leq \max_j \|w_j\| = (\|w_j\|)_{c_0}$. Otherwise, $r > 1$,
and
$$
\|\sum_j w_j\| \leq \Big(\sum_j \|w_j\|^s\Big)^{1/s} \Big(\sum_j (\Pi_{i=1}^n\|\xi_{ij}\|)^r \Big)^{1/r} \leq
 \Big(\sum_j \|w_j\|^s\Big)^{1/s} = (\|w_j\|)_s,
$$
where $1/s = 1 - 1/r = \sum{1/p_i}-n+1$.

In a similar fashion, we show that $\|\sum_j w_j\| \geq (\|w_j\|)_s$.
For $s = \infty$, the inequality $\|\sum_j w_j\| \geq \max_j \|w_j\|$ is trivial.
If $s$ is finite, 
assume $\sum_j \|w_j\|^s = 1$ (we are allowed to do so by scaling).
Find norm one vectors $\xi_{ij} \in \ell_{p_i^\prime}$
 so that $\xi_{ij} = U_{ij} \xi_i$, and
$\|w_j\| = \langle w_j, \otimes_i\xi_{ij}  \rangle$. Let $\gamma_j = \|w_j\|^{s/r}$.
Then $\sum_j \gamma_j^r = 1 = \sum_j \gamma_j \|w_j\|$. Further,
set $\alpha_{ij} = \gamma_j^{\Pi_{l\ne i} p_l^\prime/(\sum_{m=1}^n \Pi_{l \ne m} p_l^\prime )}$. An elementary calculation shows that
$\gamma_j = \Pi_{i=1}^n \alpha_{ij} $, and $\sum_j \alpha_{ij}^{p_i^\prime} = 1$.
Let $\xi_i = \sum_j \alpha_{ij} \xi_{ij}$. Then
$\|\xi\|_{p^\prime} = $, and therefore,
$$
\|\sum_j w_j\| \geq \langle \sum_j w_j, \otimes_i \xi_i \rangle =
 \sum_j\Pi_{i=1}^n \alpha_{ij} \langle w_j, \otimes_i \xi_{ij} \rangle =
 \sum_j \gamma_j \|w_j\| = 1 .
$$
This establishes the desired lower estimate.
\end{proof}

\begin{lemma}\label{l:proj est}
For $1 \leq p_i \leq \infty$,
$X=\ell_{p_1} \projtens \ell_{p_2} \projtens \ldots \projtens \ell_{p_n}$ satisfies the
$\ell_r$-estimate, where $1/r =\sum 1/p_i$ if $\sum 1/p_i<1$, and $r = 1$
otherwise. Here, we interpret $\ell_\infty$ as $c_0$.
\end{lemma}

\begin{proof}
The spaces involved all have the Contractive Projection Property
(the identity can be approximated by contractive finite rank projections).
Thus, the duality between injective and projective tensor products
of finite dimensional spaces (see e.g. \cite[Section 1.2.1]{DFS})
shows that, for $w \in X$,
$$
\|w\| = \sup \big\{ | \langle x, w \rangle | :
 x \in \ell_{p_1^\prime} \injtens \ldots \injtens \ell_{p_n^\prime} ,
\|x\| \leq 1 \big\}
$$
(here, as before, $1/p_i^\prime + 1/p_i = 1$).
Abusing the notation somewhat, we denote by $P_{im}$  the projection on the span
of the first $m$ basis vectors of both $\ell_{p_i}$ and $\ell_{p_i^\prime}$. Suppose a finite sequence $(w_k)_{k=1}^N \in X$ is block-diagonal, or more precisely,
$w_k = \big((P_{1,m_k} - P_{1,m_{k-1}}) \otimes \ldots \otimes (P_{n,m_k} - P_{n,m_{k-1}})\big) w_k$
for every $k$. Define the operator $U$ on $X$ by setting
$U x =  \sum_{k=1}^N
 \big((P_{1,m_k} - P_{1,m_{k-1}}) \otimes \ldots \otimes (P_{n,m_k} - P_{n,m_{k-1}})\big) x$.
We also use $U_0$ to denote the similarly
defined operator on $X^*$.
By ``Tong's trick'' (see e.g. \cite[p. 20]{LT1}), since $X$ and $X^*$ has an unconditional basis, $U$ ($U_0$) is a contractive projection onto its range $W$ ($W_0$).
Then
$$ \eqalign{
\|\sum_k w_k\|
&
=
\sup \big\{ |\langle \sum_k w_k, x \rangle | :
 \|x\|_{X^*} \leq 1 \big\}
\cr
&
=
\sup \big\{ |\langle U(\sum_k w_k), x \rangle | :
 \|x\|_{X^*} \leq 1 \big\}
\cr
&
=
\sup \big\{ |\langle \sum_k w_k, U_0 x \rangle | :
 \|x\|_{X^*} \leq 1 \big\} .
}  $$
Write $U_0 x = \sum_{k=1}^N x_k$.  By Lemma~\ref{l:inj est} there is an $s$
(either $1/s = \sum{1/p_i^\prime} - (n-1)=1-\sum 1/p_i$ or $s=\infty$)  $\|(\|x_k\|)\|_s =
\|U_0 x\| \leq \|x\| \leq 1$. Moreover,
$$
\langle \sum_k w_k, U_0 x \rangle =
\langle \sum_k w_k, \sum_k x_k \rangle =
\sum_k \langle w_k, x_k \rangle ,
$$
and therefore,
$$
\|\sum_k w_k\| = \sup \big\{ \sum_k |\langle w_k, x_k \rangle| : \|(x_k\|)\|_s \leq 1 \big\} =
\|(\|w_k\|)\|_r .
$$
\end{proof}

\begin{proof}[Proof of Corollary \ref{c:l_p l_q}]
Combine Theorem \ref{t:ce est} with Lemma \ref{l:inj est} and \ref{l:proj est}.
\end{proof}

\begin{proof}[Proof of Corollary \ref{c:p est}]
To apply Theorem \ref{t:ce est ops}, we have to show that
$\iK(X_1,X_2)$ satisfies the $c_0$-estimate. By renorming, we can
assume that the FDD constants of $X_1$ and $X_2$ equal $1$.
Suppose $(w_k)_{k=1}^N$ is a block-diagonal sequence, with
$w_k = (P_{2,n_k} - P_{2,n_{k-1}}) w_k (P_{1,n_k} - P_{1,n_{k-1}})$.
Let $w = \sum_k w_k$. Then
$\|w\| \geq \|(P_{2,n_k} - P_{2,n_{k-1}}) w (P_{1,n_k} - P_{1,n_{k-1}})\| = \|w_k\|$,
hence $\|w\| \geq \max_k \|w_k\|$. To prove the reverse inequality (with some constant),
pick a norm one $\xi \in X_1$, and let $\xi_k = (P_{1,n_k} - P_{1,n_{k-1}}) x$.
Then $\eta_k = w \xi_k$ satisfies $(P_{2,n_k} - P_{2,n_{k-1}}) \eta_k = \eta_k$.
Set $\eta = w \xi = \sum_k \eta_k$. Denote by $C_1$ ($C_2$) lower (upper)
$p$-estimate constants of $X_1$ (resp. $X_2$). Then
$$  \eqalign{
\|w \xi\|^p
&
=
\|\eta\|^p \leq C_1 \sum_k \|\eta_k\|^p \leq C_2 \sum_k \|w_k\|^p \|\xi_k\|^p
\cr
&
\leq
\max_k \|w_k\|_p C_2 \sum_k \|\xi_k\|^p \leq \max_k \|w_k\|^p C_2 C_1 \| \sum_k \xi_k\|^p =
C_2 C_1 \|\xi\|^p .
}  $$
Taking the supremum over all $\xi \in \ball(X_1)$,
$\|w\| \leq (C_1 C_2)^{1/p} \max_k \|w_k\|$. 
\end{proof}

In general, a tensor product of subprojective spaces
(in fact, of Hilbert spaces) need not be subprojective.

\begin{proposition}\label{p:tens not SP}
There exists a tensor norm $\otimes_\alpha$, so that, for every Banach spaces $X$ and $Y$,
$X \otimes_\alpha Y$ is a Banach space, and $\ell_2 \otimes_\alpha \ell_2$
is not subprojective.
\end{proposition}

\begin{proof}
Note first that there exists a separable symmetric sequence space $\ce$ which is not subprojective.
Indeed, let $U$ be the space with an unconditional basis which is complementably universal
for all spaces with unconditional bases, see \cite[Proposition~2.d.10]{LT1}.
As noted in \cite[Section 3.b]{LT1}, this space has
a symmetric basis (in fact, uncountably many non-equivalent symmetric bases). On the
other hand, $U$ is not subprojective, since it contains a (complemented) copy of $L_p$
for $1 < p < 2$. Renorming $U$ to make its basis $1$-symmetric, we obtain $\ce$.

Now suppose $X$ and $Y$ are Banach spaces. For $a \in X \otimes Y$, we set $\|a\|_\alpha =
\sup\{ \|(u \otimes v)(a)\|_{\ce(H,K)}\}$, where the supremum is taken
over all contractions $u : X \to H$ and $v : Y \to K$ ($H$ and $K$ are Hilbert spaces).
Clearly $\otimes_\alpha$ is a norm on $X \otimes Y$. It is easy to see that,
for any $a \in X \otimes Y$, $T_X \in B(X,X_0)$, and $T_Y \in B(Y,Y_0)$,
$\|(T_X \otimes T_Y)(a)\|_\alpha \leq \|T_X\| \|T_Y\| \|a\|_\alpha$.
Consequently, $\|x \otimes y\|_\alpha = \|x\| \|y\|$.
Thus, $\| \cdot \|_\alpha$ is indeed a tensor norm (in the sense of e.g. \cite[Section 12]{}).
We denote by $X \otimes_\alpha Y$ the completion of $X \otimes Y$ in this norm.

If $X$ and $Y$ are Hilbert spaces, then for $a \in X \otimes Y$ we have
$\|a\|_\alpha = \|a\|_{\ce(X^*,Y)}$. Identifying $\ell_2$ with its adjoint,
we see that $\ce$ embeds into $\ell_2 \otimes_\alpha \ell_2$ as the space
of diagonal operators. As $\ce$ is not subprojective, neither is
$\ell_2 \otimes_\alpha \ell_2$.
\end{proof}

Here is another wide class of non-subprojective spaces.

\begin{theorem}\label{t:B(X)}
Let $X$ be an infinite dimensional  Banach space. Then $B(X)$ is not subprojective. 
\end{theorem}

\begin{proof}
Suppose, for the sake of contradiction, that $B(X)$ is subprojective. Fix a norm one element
$x^* \in X^*$. For $x \in X$ define $T_x \in B(X) : y \mapsto \langle x^*, y \rangle x$.
Clearly $M = \{T_x : x \in X\}$ is a closed subspace of $B(X)$, isomorphic to $X$.
Therefore, $X$ is subprojective. By Proposition \ref{p:unc bas},
we can find  a subspace  $N \subset M$ with an unconditional basis.
We shall deduce that $B(X)$ contains a copy of $\ell_\infty$,
which is not subprojective.

If $N$ is not reflexive,  then $N$  contains either a copy of $c_0$ or a copy of $\ell_1$,
see \cite[Proposition~1.c.13]{LT1}. By \cite[Proposition 2.a.2]{LT1},
any subspace of $\ell_p$ ($c_0$) contains
a further subspace isomorphic to $\ell_p$ (resp. $c_0$) and complemented in $\ell_p$
(resp. $c_0$), hence we can pass from $N$ to a further subspace $W$, isomorphic
to $\ell_1$ or $c_0$, and complemented in $X$ by a projection $P$.
Embed $B(W)$ isomorphically into $B(X)$
by sending  $T\in B(W)$ to $PTP \in B(X)$, where $P$ is a projection from $X$ onto $W$.
It is easy to see that $B(W)$  contain subspaces isomorphic to $\ell_\infty$,
thus, $B(X)$ is not subprojective.

There is only one option left: $N$ is reflexive. Pick a subspace $W \subset N$, complemented in $X$.
It has the  Bounded Approximation Property \cite[Theorem~1.e.13]{LT1}.
As in the previous paragraph, $B(W)$ embeds isomorphically into $B(X)$.
Since $B(W) \ne \iK(W)$,   \cite[Theorem~4(1)]{Fed:80} shows that $B(W)$ contains
an isomorphic copy of $\ell_\infty$. This rules out the subprojectivity of $B(X)$.
\end{proof}

\begin{question}\label{q:L_p(X)}
Suppose $X$ is a subprojective Banach space.
(i) Is ${\mathrm{Rad}}(X)$ subprojective?
(ii) If $2 \leq p < \infty$, must $L_p(X)$ be subprojective?
\end{question}

\begin{question}\label{q:tensor}
Is a ``classical'' (injective, projective, etc.)
tensor product of subprojective spaces necessarily subprojective?
Note that the Fremlin tensor product $\otimes_{|\pi|}$
of Banach lattices (the ordered analogue of the projective product)
can destroy subprojectivity. Indeed, by \cite{BBPTT},
$L_2 \otimes_{|\pi|} L_2$ contains a copy of $L_1$.
$L_2$ is clearly subprojective, while $L_1$ is not (see e.g. \cite{Whi}).
\end{question}

\section{Spaces of continuous functions}\label{s:cont}

In this section we deal with spaces of functions on scattered spaces.
Recall that a topological space is \emph{scattered}
if every compact subset has an isolated point. It is known that a compact set is
scattered and metrizable if and only if it is countable (in this case, $C(K)$, and even its dual,
are separable). For more information, see e.g. \cite[Section 12]{FHHMPZ}.
It is well known that, if $K$ is a compact Hausdorff set, then
$C(K)$ is separable if and only if $K$ is metrizable.

If $K$ is countable, then 
$C(K)$ is $c_0$-saturated \cite[Section 12]{FHHMPZ}, and the copies of $c_0$
are complemented, by Sobczyk's Theorem.
Otherwise, by Milutin's Theorem
(see e.g. \cite[III.D.19]{Woj}, $C(K)$ is isomorphic to $C([0,1])$.
Thus, a separable space $C(K)$ is subprojective if and only if  $K$ is scattered.

Furthermore (see e.g. \cite{PS}), it is known that $K$ is scattered if and only if it
supports no non-zero atomic measures. Then $C(K)^*$ is isometric to $\ell_1(K)$.
Otherwise, $C(K)^*$ contains a copy of $L_1(0,1)$. Thus, $C(K)^*$ is subprojective
if and only if $K$ is scattered.

\subsection{Tensor products of $C(K)$}\label{ss:tens C(K)}

In this subsection we study the subprojectivity of projective and injective tensor
products of $C(K)$. Our main result is:

\begin{theorem}\label{t:C(K,X)}
Suppose $K$ is a compact metrizable space, and $X$ is a Banach space.
Then the following are equivalent:
\begin{enumerate}
\item 
$K$ is scattered, and $X$ is subprojective.
\item
$C(K,X)$ is subprojective.
\end{enumerate}
\end{theorem}

\begin{proof}
The implication $(2) \Rightarrow (1)$ is easy.
The space $C(K,X)$ contains copies of $C(K)$ and of $X$,
hence the last two spaces are subprojective. By the preceding paragraph,
$K$ must be scattered.

To prove $(1) \Rightarrow (2)$, first fix some notation.
Suppose $\lambda$ is a countable ordinal.
We consider the interval $[0,\lambda]$ with the order topology --
that is, the topology generated by the open intervals $(\alpha,\beta)$,
as well as $[0,\beta)$ and $(\alpha,\lambda]$. Abusing the notation
slightly, we write $C(\lambda,X)$ for $C([0,\lambda],X)$.

Suppose $K$ is scattered. By \cite[Chapter 8]{Se71}, $K$ is isomorphic
to $[0,\lambda]$, for some countable limit ordinal $\lambda$. Fix a
subprojective space $X$. We use induction on $\lambda$ to show that,
for any countable ordinal $\lambda$,
\begin{equation}
C(\lambda,X)  \textrm{  is subprojective}.
\label{eq:subp}
\end{equation}
By Proposition \ref{p:dir_sum}, \eqref{eq:subp} holds for $\lambda \leq
\omega$ (indeed, $c$ is isomorphic to $c_0$, hence
$c(X) = c \injtens X$ is isomorphic to $c_0(X) = c_0 \injtens X$).
Let ${\mathcal{F}}$ denote the set of all countable ordinals for
which \eqref{eq:subp} fails. If ${\mathcal{F}}$ is non-empty, then
it contains a minimal element, which we denote by $\mu$.
Note that $\mu$ is a limit ordinal. Indeed, otherwise it has an
immediate predecessor $\mu-1$. It is easy to see that
$C(\mu,X)$ is isomorphic to $C(\mu-1,X) \oplus X$, hence, by
Proposition \ref{p:dir_sum}, $C(\mu-1,X)$ is not subprojective.
Let $C_0(\mu,X) = \{f \in C(\mu,X) : \lim_{\nu \to \mu} f(\nu) = 0\}$.
Clearly $C(\mu,X)$ is isomorphic to $C_0(\mu,X) \oplus X$, hence
we obtain the desired contradiction by showing that $C_0(\mu,X)$ is subprojective.

To do this, suppose $Y$ is a subspace of $C_0(\mu,X)$, so that
no subspace of $Y$ is complemented in $C_0(\mu,X)$. For $\nu < \mu$,
define the projection $P_\nu : C(\mu,X) \to C(\nu,X) : f \mapsto f \one_{[0,\nu]}$.
If, for some $\nu < \mu$ and some subspace $Z \subset Y$, $P_\nu|_Z$
is an isomorphism, then $Z$ contains a subspace complemented in $X$, by
the induction hypothesis and Corollary~\ref{c:complem}. Now suppose
$P_\nu|_Y$ is strictly singular for any $\nu$. We construct a sequence
of ``almost disjoint'' elements of $Y$. To do this, take an arbitrary
$y_1$ from the unit sphere of $Y$. Pick $\nu_1 < \mu$ so that
$\|y_1 - P_{\nu_1} y_1\| < 10^{-1}$. Now find a norm one $y_2 \in Y$
so that $\|P_{\nu_1} y_2\| < 10^{-2}/2$. Proceeding further in the same
manner, we find a sequence of ordinals $0 = \nu_0 < \nu_1 < \nu_2 < \ldots$, and
a sequence of norm one elements $y_1, y_2, \ldots \in Y$, so that
$\|y_k - z_k\| < 10^{-k}$, where $z_k = (P_{\nu_k} - P_{\nu_{k-1}}) y_k$.
The sequence $(z_k)$ is equivalent to the $c_0$ basis, and the same is true
for the sequence $(y_k)$.

Moreover, $\span[z_k : k \in \N]$ is complemented in $C(\mu,X)$. Indeed,
let $\nu = \sup_k \nu_k$. We claim that $\mu = \nu$. If $\nu < \mu$,
then $P_\nu$ is an isomorphism on $\span[y_k : k \in \N]$, contradicting
our assumption. Let $W_k = (P_{\nu_k} - P_{\nu_{k-1}})(C_0(X))$, and find
a norm one linear functional $w_k$ so that $w_k(z_k) = \|z_k\|$. Define
$$
Q : C_0(\mu,X) \to C_0(\mu,X) : f \mapsto
  \sum_k w_k \big( (P_{\nu_k} - P_{\nu_{k-1}}) f \big) z_k .
$$
Note that $\lim_k \|(P_{\nu_k} - P_{\nu_{k-1}}) f\| = 0$, hence
the range of $Q$ is precisely the span of the elements $z_k$.
By Small Perturbation Principle, $Y$ contains a subspace complemented
in $C_0(\mu,X)$.
\end{proof}

The above theorem shows that $C(K) \injtens X$ is subprojective if and only if both
$C(K)$ and $X$ are. We do no know whether a similar result holds for
other tensor products. We do, however, have:

\begin{proposition}\label{p:C(K) projtens}
Suppose $K$ is a compact metrizable space, and $W$ is either $\ell_p$ ($1 \leq p < \infty$) or $c_0$.
Then $C(K) \projtens W$ is subprojective if and only if $K$ is scattered.
\end{proposition}

\begin{proof}
Clearly, if $K$ is not scattered, then $C(K)$ is not subprojective.
So suppose $K$ is scattered. 
We deal with the case of $W = \ell_p$, as the $c_0$ case is handled similarly.
As before, we can assume that $K = [0,\lambda]$, where $\lambda$ is a countable ordinal.
We use transfinite induction on $\lambda$. The base case is easy: if $\lambda$ is a finite
ordinal, then $C(\lambda) \projtens \ell_p = \ell_\infty^N \projtens \ell_p$ is subprojective.
Furthermore the same is true for $\lambda = \omega$ (then $C(\lambda) = c$).

Suppose, for the sake of contradiction, that $\lambda$ is the smallest countable ordinal so that
$C(\lambda) \projtens \ell_p$ is not subprojective. Reasoning as before, we conclude that
$\lambda$ is a limit ordinal. Furthermore, $C(\lambda) \sim C_0(\lambda)$, hence
$C_0(\lambda) \projtens \ell_p$ is not subprojective.

Denote by $Q_n : \ell_p \to \ell_p$ the projection on the first $n$ basis vectors in $\ell_p$,
and let $Q_n^\perp = I - Q_n$. For $f \in C_0(\lambda)$ and an ordinal $\nu < \lambda$,
define $P_\nu f = \chi_{[0,\nu]} f$, and $P_\nu^\perp = I - P_\nu$.

Suppose $X$ is a subspace of $C_0(\lambda) \projtens \ell_p$ which has no subspaces
complemented in $C_0(\lambda) \projtens \ell_p$. By the induction hypothesis,
$(P_\nu \otimes I_{\ell_p})|_Y$ is strictly singular for any $\nu < \lambda$.
Furthermore, $(I_{C_0(\lambda)} \otimes Q_n)|_Y$ must be strictly singular.
Indeed, otherwise $Y$ has a subspace $Z$ so that $(I_{C_0(\lambda)} \otimes Q_n)|_Z$
is an isomorphism, whose range is subprojective (the range of $I_{C_0(\lambda)} \otimes Q_n$
is isomorphic to the sum of $n$ copies of $C(\lambda)$, hence subprojective).
Therefore, for any $\nu < \lambda$ and $n \in \N$,
$(I - P_\nu^\perp \otimes Q_n^\perp)|_Y$ is strictly singular.
Therefore we can find a normalized basis $(x_i)$ in $Y$, and sequences
$0 = \nu_0 < \nu_1 < \ldots < \lambda$, and $0 = n_0 < n_1 < \ldots$,
so that $\|x_i - (P_{\nu_{i-1}}^\perp \otimes Q_{n_{i-1}}^\perp) x_i\| < 10^{-3i}/2$.
By passing to a further subsequence, we can assume that
$\|(P_{\nu_i} \otimes Q_{n_i})x_i\| < 10^{-3i}/2$. Thus, by  the Small Perturbation Principle,
it suffices to show the following statement:
If $(y_i)$ is a normalized sequence is $C_0(\lambda) \projtens \ell_p$,
so that there exist non-negative integers $0 = n_0 < n_1 < n_2 < \ldots$, and
ordinals $0 = \nu_0 < \nu_1 < \nu_2 < \ldots < \lambda$, with the property that
$y_i = \big((P_{\nu_i} - P_{\nu_{i-1}}) \otimes (Q_{n_i} - Q_{n_{i-1}})\big) y_i$
for any $i$, then $Y = \span[y_i : i \in \N]$ is contractively complemented in $C(K) \projtens \ell_p$.

Denote by $X$ the span of all $x$'s for which there exists an $i$ so that
$x = \big((P_{\nu_i} - P_{\nu_{i-1}}) \otimes (Q_{n_i} - Q_{n_{i-1}})\big) x$.
Then $Y$ is contractively complemented in $C(K) \projtens \ell_p$. In fact, we
can define a contractive projection onto $X$ as follows. Suppose first
$u = \sum_{j=1}^N a_j \otimes b_j$, with $b_i$'s having finite support in $\ell_p$.
Then set $P u = \sum_{i=1}^\infty \big((P_{\nu_i} - P_{\nu_{i-1}}) \otimes (Q_{n_i} - Q_{n_{i-1}})\big) u$.
Due to our assumption on the $b_i$'s, there exists $M$ so that
$P u = \sum_{i=1}^M \big((P_{\nu_i} - P_{\nu_{i-1}}) \otimes (Q_{n_i} - Q_{n_{i-1}})\big) u$.
To show that $\|Pu\| \leq \|u\|$, define, for $\vr = (\vr_i)_{i=1}^M \in \{-1,1\}^M$,
the operator of multiplication by $\sum_{i=1}M \vr_i \chi_{[\nu_{i-1}+1,\nu_i]}$
on $C_0(\lambda)$. The operator $V_\vr \in B(\ell_p)$ is defined similarly.
Bot $U_\vr$ and $V_\vr$ are contractive. Furthermore, $Pu = \ave_\vr (U_\vr \otimes V_\vr) u$.
Therefore, we can use continuity to extend $P$ to a contractive projection from
$C_0(\lambda) \projtens \ell_p$ onto $X$.

It To construct a contractive projection from $X$ onto $Y$, we need to show that
the blocks of $X$ satisfy the $\ell_p$-estimate. That is, if
$x_i = \big((P_{\nu_i} - P_{\nu_{i-1}}) \otimes (Q_{n_i} - Q_{n_{i-1}})\big) x_i$
for each $i$, then $\|\sum_i x_i\|^p = \sum_i \|x_i\|^p$. To this end, use trace duality
to identify $(C_0(\lambda) \projtens \ell_p)^*$ with $B(\ell_p, \ell_1([0,\lambda)),)$.
$P^*$ is the ``block'' projection onto the space of ``block diagonal''
operators which map the elements of $\ell_p$ supported on $(n_{i-1}, n_i]$
onto the vectors in $\ell_1$ supported on $(\nu_{i-1}, \nu_i]$. If $T_i's$ are the blocks
of such an operator, then $\|\sum_i T_i\|^{p^\prime} = \sum_i \|T_i\|^{p^\prime}$, where
$1/p + 1/p^\prime = 1$. (see the proof of Corollary \ref{c:l_p l_q}). By duality,
$\|\sum_i x_i\|^p = \sum_i \|x_i\|^p$.
\end{proof}

\begin{remark}
Suppose $K$ is a scattered metrizable space. We do not know whether $C(K) \projtens C(K)$
is necessarily subprojective. The proof above cannot be emulated directly, since $P$
may not be well-defined. More specifically, we cannot quite define $P u$ if $u = f \otimes f$,
with $f = \chi_{[0,\mu]}$, with $\sup_i \nu_i < \mu < \lambda$.
\end{remark}

\subsection{Operators on $C(K)$}\label{ss:op on C(K)}

\begin{proposition}\label{p:P pq}
Suppose $K$ is a scattered compact metrizable space, and $1 \leq p \leq q < \infty$.
Then the space $\Pi_{qp}(C(K),\ell_q)$ is subprojective.
\end{proposition}

Recall that $\Pi_{qp}(X,Y)$ stands for the space of $(q,p)$-summing operators -- that is,
the operators for which there exists a constant $C$ so that, for any $x_1, \ldots, x_n \in X$,
$$
\Big(\sum_i \|T x_i\|^q\Big)^{1/q} \leq C \sup_{x^* \in \ball(X^*)}
 \Big( \sum_i | x^*(x_i) |^p \Big)^{1/p} .
$$
The smallest value of $C$ is denoted by $\pi_{pq}(T)$.

Note that, if a compact Hausdorff space $K$ is not scattered, then $C(K)^*$ contains $L_1$ 
\cite{PS}, hence $\Pi_{qp}(C(K),\ell_q)$ is not subprojective.

The following lemma may be interesting in its own right.

\begin{lemma}\label{l:P pq}
Suppose $X$ is a Banach space, $K$ is a compact metrizable scattered space, and
$1 \leq p \leq q < \infty$. Then, for any $T \in \Pi_{qp}(C(K),X)$, and any $\vr > 0$,
there exists a finite rank operator $S \in \Pi_{qp}(C(K),X)$ with
$\pi_{pq}(T-S) < \vr$.
\end{lemma}

In proving Proposition \ref{p:P pq} and Lemma \ref{l:P pq}, we consider the cases of $p=q$
and $p<q$ separately. If $p=q$, we are dealing with $q$-summing operators. By Pietsch
Factorization Theorem, $T \in B(C(K),X)$ is $q$-summing if and only if  there exists a probability measure
$\mu$ on $K$ so that $T$ factors as $\tilde{T} \circ j$, where $j : C(K) \to L_q(\mu)$
is the formal identity, and $\|T\| \leq \pi_q(T)$. Moreover, $\mu$ and $\tilde{T}$
can be selected in such a way that $\|\tilde{T}\| = \pi_q(T)$. As $K$ is scattered, there exist
distinct points $k_1, k_2, \ldots \in K$, and non-negative scalars $\alpha_1, \alpha_2, \ldots$,
so that $\sum_i \alpha_i = 1$, and $\mu= \sum_i \alpha_i \delta_{k_i}$.

Now suppose $T \in B(C(K),X)$ satisfies $\pi_q(T) = 1$. Keeping the above notation, find
$N \in \N$ so that $(\sum_{i=N+1}^\infty \alpha_i)^{\frac{1}{q}} < \vr$. Denote by $u$ and $v$ the
operators of multiplication by $\chi_{\{k_1, \ldots, k_N\}}$ and
$\chi_{\{k_{N+1}, k_{N+2}, \ldots\}}$, respectively, acting on $L_q(\mu)$. It is easy
to see that $\rank u \leq N$, and $\|v j\| < \vr$. Then $S = \tilde{T} u j$
works in Lemma \ref{l:P pq}.

If $1 \leq p < q$, then (see e.g. \cite[Chapter 10]{DJT} or \cite[Chaper 21]{TJ}),
$\Pi_{qp}(C(K),X) = \Pi_{q1}(C(K),X)$, with equivalent norms. Henceforth, we set $p=1$.
We have a probability measure $\mu$ on $K$, and a factorization $T = \tilde{T} j$, where
$j : C(K) \to L_{q1}(\mu)$ is the formal identity, and $\tilde{T} : L_{q1}(\mu) \to X$
satisfies $\|\tilde{T}\| \leq c \pi_{q1}(T)$ ($c$ is a constant depending on $q$).

In this case, the proof of Lemma \ref{l:P pq} proceeds as for $q$-summing operators, except that
now, we need to select $N$ so that $c \big(\sum_{i=N+1}^\infty \alpha_i\big)^{1/q} < \vr$.

\begin{proof}[Proof of Proposition \ref{p:P pq}]
It is well known that, for any $T$, $\pi_{qp}(T) = \pi_{qp}(T^{**})$.
Moreover, by Lemma \ref{l:P pq}, any $(q,p)$-summing operator on $C(K)$ can be
approximated by a finite rank operator. Then we can identify $\Pi_{qp}(C(K),X)$
with the completion of the algebraic tensor product $C(K)^* \otimes X$ in the
appropriate tensor norm which we denote by $\alpha$. Recalling that $C(K)^* = \ell_1$
(the canonical basis in $\ell_1$ corresponds to the point evaluation functionals),
we can describe $\alpha$ in more detail: for $u = \sum_i a_i \otimes x_i \in \ell_1 \otimes X$,
$\|u\|_\alpha = \pi_{qp}(\overline{u})$, where $\overline{u} : \ell_\infty \to X$ is defined by
$\overline{u} b = \sum_i b(a_i) x_i$. Furthermore, by the injectivity of the ideal $\Pi_{qp}$,
$\pi_{qp}(\overline{u}) = \pi_{qp}(\kappa_X \circ \overline{u})$, where $\kappa_X : X \to X^{**}$
is the canonical embedding. Finally, $\kappa_X \circ \overline{u} = \tilde{u}^{**}$, with
$\tilde{u} : c_0 \to X$ defined via $\tilde{u} b = \sum_i b(a_i) x_i$.

To finish the proof, we need to show (in light of Theorem \ref{t:ce est}) that
$\ell_1 \otimes_\alpha \ell_q$ satisfies the $\ell_q$ estimate. To this end, suppose
we have a block-diagonal sequence $(u_i)_{i=1}^n$, and show that $\| \sum_i u_i \|_\alpha^q \sim
\sum_i \| u_i \|_\alpha^q$. Abusing the notation slightly, we identify $u_i$
with an operator from $\ell_\infty^N$ to $\ell_q^N$ (where $N$ is large enough),
and identify $\| \cdot \|_\alpha$ with $\pi_{qp}(\cdot)$.

First show that $\| \sum_i u_i \|_\alpha^q \leq c^q \sum_i \| u_i \|_\alpha^q$,
where $c$ is a constant (depending on $q$). We have disjoint sets $(S_i)_{i=1}^n$
in $\{1, \ldots, N\}$ so that $u_i e_j = 0$ for $j \notin S_i$. Therefore there
exists a probability measure $\mu_i$, supported on $S_i$,
so that
$$
\|u_i f\|^q \leq c_1^q \pi_{qp}(u_i)^q \|f\|_\infty^{q-p} \|f\|_{L_p(\mu_i)}^p
$$
for any $f \in \ell_\infty^N$ ($c_1$ is a constant). Now define the probability measure
$\mu$ on $\{1, \ldots, N\}$:
$$
\mu = \big( \sum_i \pi_{qp}(u_i)^q \big)^{-1} \sum_i \pi_{qp}(u_i)^q \mu_i .
$$
For $f \in \ell_\infty^N$, set $f_i = f \chi_{S_i}$. Then the vectors $u_i f_i$ are
disjointly supported in $\ell_q$, and therefore,
$$
\|(\sum_i u_i)f\|^q = \sum_i \|u_i f_i\|^q \leq
c_1^q \sum_i \pi_{qp}(u_i)^q \|f_i\|_\infty^{q-p} \|f_i\|_{L_p(\mu_i)}^p \leq
c_1^q \|f\|_\infty^{q-p} \sum_i \pi_{qp}(u_i)^q \|f_i\|_{L_p(\mu_i)}^p .
$$
An easy calculation shows that
$$
\|f_i\|_{L_p(\mu_i)}^p = \big( \sum_i \pi_{qp}(u_i)^q \big)^{-1} \sum_i \pi_{qp}(u_i)^q \|f_i\|_{L_p(\mu)}^p ,
$$
hence
$$
\|(\sum_i u_i)f\|^q \leq
c_1^q \big( \sum_i \pi_{qp}(u_i)^q \big) \|f\|_\infty^{q-p} \sum_i \|f_i\|_{L_p(\mu_i)}^p =
c_1^q \big( \sum_i \pi_{qp}(u_i)^q \big) \|f\|_\infty^{q-p} \|f\|_{L_p(\mu)}^p .
$$
Therefore, $\pi_{qp}(\sum_i u_i) \leq c \big( \sum_i \pi_{qp}(u_i)^q \big)^{1/q}$,
for some universal constant $c$.

Next show that $\| \sum_i u_i \|_\alpha^q \geq c^{\prime q} \sum_i \| u_i \|_\alpha^q$,
where $c^\prime$ is a constant. There exists a probability measure $\mu$ on $\{1, \ldots, N\}$
so that, for any $f \in \ell_\infty^N$,
$$
\|(\sum_u u_i f)\|^q \geq c_2^q \pi_{qp}(\sum_i u_i)^q \|f\|_\infty^{q-p} \|f\|_{L_p(\mu)}^p
$$
For each $i$ let $\alpha_i = \|\mu|_{S_i}\|_{\ell_1^N}$, and $\mu_i = \mu_i/\alpha_i$
(if $\alpha_i = 0$, then clearly $u_i = 0$). Then for any $i$, and any $f \in \ell_\infty^N$,
$$
\|u_i f\|^q = \|(\sum_i u_i)(\chi_{S_i} f)\|^q
\leq c_2^q \pi_{qp}(\sum_i u_i)^q \alpha_i \|f\|_\infty^{q-p} \|f\|_{L_p(\mu_i)}^p ,
$$
hence $\pi_{qp}(u_i) \leq c^\prime \alpha_i^{1/q} \pi_{qp}(\sum_i u_i)$ ($c^\prime$ is a constant).
As $\sum_i \alpha_i = 1$, we conclude that
$\sum_i \pi_{qp}(u_i)^q \leq c^{\prime q} \pi_{qp}(\sum_i u_i)$.
\end{proof}

\subsection{Continuous fields}\label{ss:fields}
We refer the reader to \cite[Chapter 10]{Dix} for an introduction into
continuous fields of Banach spaces. To set the stage,
suppose $K$ is a locally compact Hausdorff space (the \emph{base space}), and
 $(X_t)_{t \in K}$ is a family of Banach spaces (the spaces $X_t$ are called
(\emph{fibers}).
A \emph{vector field} is
an element of $\prod_{t \in K} X_t$.  A linear subspace $X$ of $\prod_{t \in K} X_t$ is called
a \emph{continuous field} if the following conditions hold:
\begin{enumerate}
\item 
For any $t \in K$, the set $\{x(t) : x \in X\}$ is dense in $X_t$.
\item
For any $x \in X$, the map $t \mapsto \|x(t)\|$ is continuous, and vanishes at infinity.
\item
Suppose $x$ is a vector field so that, for any $\vr > 0$ and any $t \in K$,
there exist an open neighborhood $U \ni t$ and $y \in X$ for which $\|x(s) - y(s)\| < \vr$
for any $s \in U$. Then $x \in X$.
\end{enumerate}
Equipping $X$ with the norm $\|x\| = \max_t \|x(t)\|$, we turn it into a Banach space.

In a fashion similar to Theorem \ref{t:C(K,X)}, we prove:

\begin{proposition}\label{p:field}
Suppose $K$ is a scattered metrizable space, $X$ is a separable continuous
vector filed on $K$, so that, for every $t \in K$, the fiber $X_t$ is subprojective.
Then $X$ is subprojective.
\end{proposition}

\begin{proof}
Using one-point compactification if necessary (as in \cite[10.2.6]{Dix}),
we can assume that $K$ is compact.
As before, we assume that $K = [0,\lambda]$ ($\lambda$ is a countable ordinal).
We denote by $X_{(0)}$ the set of all $x \in X$ which vanish at $\lambda$.
If $\nu \leq \lambda$, we denote by $X_{[\nu]}$ the set of all $x \in X_\lambda$
which vanish outside of $[0,\nu]$. By \cite[Proposition 10.1.9]{Dix},
$x \chi_{[0,\nu]} \in X$ for any $x \in X$, hence $X_{[\nu]}$ is a Banach space.
We then define the restriction operator $P_\nu : X \to X_{[\nu]}$.
We denote by $Q_\nu : X \to X_\nu$ the operator of evaluation at $\nu$.

We say that a countable ordinal $\lambda$ has Property ${\mathcal{P}}$ if,
whenever $X$ is a continuous separable vector field whose fibers are
subprojective, then $X$ is subprojective. Using transfinite induction,
we prove that any countable ordinal has this property.

The base of induction is easy to handle. Indeed, when $\lambda$ is finite, then
$X$ embeds into a direct sum of (finitely many) subprojective spaces $X_\nu$.
Now suppose, for the sake of contradiction, that $\lambda$ is the smallest
ideal failing Property ${\mathcal{P}}$. Note that $\lambda$ is a limit ordinal.
Indeed, otherwise it has an immediate predecessor $\lambda_-$, and $X$ embeds
into a direct sum of two subprojective spaces -- namely, $X_{[\lambda_-]}$ and $X_\lambda$.

Suppose $Y$ is a subspace of $X$, so that no subspace of $Y$ is complemented in $X$.
We shall achieve a contradiction once we show that $Y$ contains a copy of $c_0$.

By Proposition \ref{p:compare}, $Q_\lambda$ is strictly singular on $Y$.
Passing to a smaller subsequence if necessary, we can assume that, $Y$ has a basis
$(y_i)_{i \in \N}$, so that (i) for any finite sequence
$(\alpha_i)$, $\|\sum_i \alpha_i y_i\| \geq \max_i |\alpha_i|/2$,
and (ii) for any $i$, $\|Q_\lambda y_i\| < 10^{-4i}$. 
Consequently, for any $y \in \span[y_j : j > i]$, $\|Q_\lambda y\| < 10^{-4i}$.
Indeed, we can assume that $y$ is a norm one vector with finite support, and write
$y$ as a finite sume $y = \sum_j \alpha_j y_j$. By the above, $|\alpha_i| \leq 2$
for every $i$. Consequently, $\|Q_\lambda y\| \leq \sum_j |\alpha_j| \|Q_\lambda y_j\|
\leq 2 \sum_{j > i} 10^{-4j} < 10^{-4i}$.

Now construct a sequence $\nu_1 < \nu_2 < \ldots < \lambda$ of ordinals,
a sequence $1 = n_1 < n_2 < \ldots$ or positive integers, and a sequence $x_1, x_2, \ldots$
of norm one vectors, so that (i) $x_j \in \span[y_i : n_j \leq i < n_{j+1}]$,
(ii) $\|P_{\nu_i} x_i\| < 10^{-4i}$, and (iii) $\|P_{\nu_{i+1}} x_i\| < 10^{-4i}$.
To this end, recall that, by Proposition \ref{p:compare} again, $P_\nu|_Y$ is strictly singular
for any $\nu < \lambda$. Pick an arbitrary $\nu_1 < \lambda$, and find a norm $1$ vector
$x_1 \in \span[y_1, \ldots, y_{n_2-1}]$ so that $\|P_{\nu_1} x_1\| < 10^{-4}$.
We have $\|Q_\lambda x_1\| < 10^{-4}$. By continuity, we can find $\nu_2 > \nu_1$ so that
$\|P_{\nu_2} x_1\| < 10^{-4}$. Next find a norm one $x_2 \in \span[y_{n_2}, \ldots, y_{n_3-1}]$
so that $\|P_{\nu_2} x_1\| < 10^{-8}$. Proceed further in the same manner.

We claim that the sequence $(x_i)$ is equivalent to the canonical basis in $c_0$.
Indeed, for each $i$ let $x_i^{\prime\prime} = P_{\nu_i} x_i + P_{\nu_{i+1}} x_i$,
and $x_i^\prime = x_i - x_i^{\prime\prime}$. Since we are working with the $\sup$ norm,
$\|x_i^\prime\| = \|x_i\| = 1$ for any $i$. Furthermore, the elements $x_i^\prime$ are
disjointly supported, hence, for any $(\alpha_i)$ finite sequence of scalars $(\alpha_i)$,
$\|\sum_i \alpha_i x_i^\prime\| = \max_i |\alpha_i|$. By the triangle inequality,
$$
\Big| \|\sum_i \alpha_i x_i\| - \|\sum_i \alpha_i x_i^\prime\| \Big| \leq
\sum_i |\alpha_i| \|x_i^{\prime\prime}\| < \max_i |\alpha_i| \sum_{i=1}^\infty 2 \cdot 20^{-4i} <
10^{-3} \max_i |\alpha_i| ,
$$
which yields the desired result.
\end{proof}

To state a corollary of Proposition \ref{p:field},
recall that a $C^*$-algebra $\A$ is \emph{CCR} (or \emph{liminal}) if,
for any irreducible representation $\pi$ of $\A$ on a Hilbert space $H$,
$\pi(\A) = \iK(H)$. A $C^*$-algebra $\A$ is \emph{scattered} if every
positive linear functional on $\A$ is a sum of pure linear functionals
($f \in \A^*$ is called \emph{pure} if it belongs to an extreme ray of
the positive cone of $\A^*$). For equivalent descriptions of scattered
$C^*$-algebras, see e.g. \cite{Je, Je78, Ku}.

\begin{corollary}\label{c:field}
Any separable scattered CCR $C^*$-algebra is subprojective.
\end{corollary}

\begin{proof}
Suppose $\A$ is a separable scattered CCR $C^*$-algebra.
As shown in \cite[Sections 6.1-3]{Ped}, the spectrum of a separable CCR algebra
is a locally compact Hausdorff space. If, in addition, the algebra is scattered,
then its spectrum $\hat{\A}$ is scattered as well \cite{Je, Je78}.
In fact, by the proof of \cite[Theorem 3.1]{Je}, $\hat{\A}$ is separable.
It is easy to see that any separable locally compact Hausdorff space
is metrizable.

By \cite[Section 10.5]{Dix}, $\A$ can be represented as a vector field over
$\hat{\A}$, with fibers of the form $\pi(\A)$, for irreducible representations $\pi$.
As $\A$ is CCR, the spaces $\pi(\A) = \iK(H_\pi)$ ($H_\pi$ being a separable Hilbert space)
are subprojective. To finish the proof, apply Proposition \ref{p:field}.
\end{proof}

The last corollary leads us to

\begin{conjecture}\label{con:C*}
A separable $C^*$-algebra is scattered if and only if it is subprojective.
\end{conjecture}

It is known (\cite{Je78}, see also \cite{Ku}) that a scattered $C^*$-algebra is GCR.
However, it need not be CCR (consider the unitization of $\iK(\ell_2)$).

\section{Subprojectivity of Schatten spaces}\label{s:schatten}

In this section, we establish:

\begin{proposition}\label{p:sch_subpr}
Suppose $\ce$ is a symmetric sequence space, not containing $c_0$.
Then $\se$ is subprojective if and only if $\ce$ is subprojective.
\end{proposition}

The assumptions of this proposition are satisfied, for instance, if $\ce = \ell_p$
($1 \leq p < \infty$), or if $\ce$ is the Lorentz space $\lore(w,p)$
(see \cite[Proposition 4.e.3]{LT1}. However, not every symmetric sequence space
is subprojective. Indeed, suppose $\ce$ is Pelczynski's universal space: it
has an unconditional basis $(u_i)$ so that any other unconditional basis is equivalent
to its subsequence. As explained in \cite[Section 3.b]{LT1}, $\ce$ has a symmetric basis.
Fix $1 < p < q < 2$. Then the Haar basis in $L_p(0,1)$ is unconditional, hence
$L_p(0,1)$ is isomorphic to a complemented subspace $X$ of $\ce$. It is well known that
$\ell_q$ is contained in $L_p(0,1)$. Call the corresponding subspace of $\ce$ by
$X^\prime$. Then no subspace of $X^\prime$ is complemented in $\ce$: otherwise,
$L_p(0,1)$ would contain a complemented copy of $\ell_q$, which is impossible.

For the proof, we need a technical result.

\begin{proposition}\label{p:w null no c0}
Suppose $\se$ is a symmetric sequence space, not containing $c_0$.
Suppose, furthermore, that $(z_n) \subset \se$ is a normalized sequence,
so that, for every $k$, $\lim_n \|Q_k z_n\| = 0$. Then, for any $\vr > 0$,
$\se$ contains sequences $(\tilde{z}_n)$ and $(z_n^\prime)$, so that:
\begin{enumerate}
\item 
$(\tilde{z}_n)$ is a subsequence of $(z_n)$.
\item
$\sum_n \|\tilde{z}_n - z_n^\prime\| < \vr$.
\item
$(z_n^\prime)$ lies in the subspace $Z$ of $\se$, with the property that
{\rm{(i)}} $Z$ is $3$-isomorphic to either $\ell_2$, $\ce$, or $\ell_2 \oplus \ce$, and
{\rm{(ii)}} $Z$ is the range of a projection of norm not exceeding $3$.
\end{enumerate}
\end{proposition}

\begin{proof}
\cite[Corollary 2.8]{Ar} implies the existence of $(\tilde{z}_n)$ and $(z_n^\prime)$,
so that (1) and (2) are satisfied, and
$z_k^\prime = a \otimes E_{1k} + b \otimes E_{k1} + c_k \otimes E_{kk}$ ($k \geq 2$).
Thus, $z_n^\prime \subset Z = Z_r + Z_c + Z_d$, where
$Z_r = \span[a \otimes E_{1k} : k \geq 2]$ (the row component),
$Z_c = \span[b \otimes E_{k1} : k \geq 2]$ (the column component), and
$Z_d$ (the diagonal component) contains $c_k \otimes E_{kk}$, for any $k$.
More precisely, we can write $c_k = u_k d_k v_k$, where $u_k$ and $v_k$ are
unitaries, and $d_k$ is diagonal. Then we set
$Z_d = \span[u_k E_{ii} v_k \otimes E_{kk} : i \in \N, k \geq 2]$.

It remains to build contractive projections $P_r$, $P_c$, and $P_d$
onto $Z_r$, $Z_c$, and $Z_d$, respectively, so that $Z_c \cup Z_d \subset \ker P_r$,
$Z_r \cup Z_d \subset \ker P_c$, and $Z_r \cup Z_c \subset \ker P_d$.
Indeed, then $P = P_r + P_c + P_d$ is a projection onto $Z_r + Z_c + Z_d$, and
the latter space is completely isomorphic to $Z_0 = Z_r \oplus Z_c \oplus Z_d$.
The spaces $Z_r$, $Z_c$, and $Z_d$ are either trivial (zero-dimensional), or
isomorphic to $\ell_2$, $\ell_2$, and $\ce$, respectively.

$P_d$ is nothing but a coordinate projection, in the appropriate basis:
$$
P_d\big(u_k E_{ij} v_\ell \otimes E_{k\ell}\big) = \left\{ \begin{array}{cc}
   u_k E_{ii} v_k  \otimes E_{kk}  &  k = \ell \geq 2,  i = j   \\
   0  &   {\mathrm{otherwise}}
\end{array} \right.
$$
(for the sake of convenience, we set $u_1 = v_1 = I_{\ell_2}$).
Next construct $P_r$ ($P_c$ is dealt with similarly).
If $a=0$, just take $P_r = 0$. Otherwise, let $a^\prime = a/\|a\|$, and
find $f \in \se^*$ so that $\|f\| = 1 = \langle f, a^\prime \rangle$.
For $x = \sum_{k,\ell} x_{k\ell} \otimes E_{k\ell}$, define
$$
P_r x = a^\prime \otimes \sum_{\ell \geq 2} \langle f, x_{1\ell} \rangle E_{1\ell} ,
$$
hence $\|P_r x\|_\ce^2 = \sum_{\ell \geq 2} |\langle f, x_{1\ell} \rangle|^2$.
It remains to show $\|P_r x\| \leq \|x\|$. This inequality is obvious
when $P_r x = 0$. Otherwise, set, for $\ell \geq 2$,
$$
\alpha_\ell = \frac{\overline{\langle f, x_{1\ell} \rangle}}{
  (\sum_{\ell \geq 2} |\langle f, x_{1\ell} \rangle|^2)^{1/2}} ,
$$
$y = I_{\ell_2} \otimes \sum_{\ell \geq 2} \alpha_\ell E_{\ell 1}$, and
$z = I_{\ell_2} \otimes E_{11}$. Then
$\|y\|_\infty = \big( \sum_{\ell \geq 2} |\alpha_\ell|^2 \big)^{1/2} = 1 = \|z\|_\infty$,
and $zxy = \sum_{\ell \geq 2} \alpha_\ell x_{1\ell} \otimes E_{11}$.
Therefore,
$$  \eqalign{
\|P_r x\|_\ce
&
=
\big\langle f, \sum_{\ell \geq 2} \alpha_\ell x_{1\ell} \big\rangle \leq
\big\|\sum_{\ell \geq 2} \alpha_\ell x_{1\ell}\big\|_\ce
\cr
&
=
\big\|\sum_{\ell \geq 2} \alpha_\ell x_{1\ell} \otimes E_{11}\big\|_\ce =
\|zxy\|_\ce \leq \|z\|_\infty \|x\|_\ce \|y\|_\infty = \|x\|_\ce ,
}  $$
which is what we need.
\end{proof}

\begin{proof}[Proof of Proposition \ref{p:sch_subpr}]
The space $\se$ contains an isometric copy of $\ce$, hence the subprojectivity
of $\se$ implies that of $\ce$. To prove the converse,
suppose $\ce$ is subprojective, and $Z_0$ is a subspace of $\se$, and show that
it contains a further subspace $Z$, complemented in $\se$.
To this end, find a normalized sequence $(z_n) \subset Z_0$, so that
$\lim_n \|Q_k z_n\| = 0$ for every $k$. By Proposition \ref{p:w null no c0},
$(z_n)$ has a subsequence $(z_n^\prime)$, contained in a subspace $Z_1$,
which is complemented in $\se$, and isomorphic either to $\ce$, $\ell_2$,
or $\ce \oplus \ell_2$. By Proposition \ref{p:dir_sum}, $Z_1$ is subprojective,
hence $\span[z_n^\prime : n \in \N]$ contains a subspace complemented in
$Z_1$, hence also in $\se$.
\end{proof}

As a consequence we obtain:

\begin{proposition}\label{p:predual}
The predual of a von Neumann algebra $\A$ is subprojective if and only if $\A$
is purely atomic.
\end{proposition}

We say that $\A$ is \emph{purely atomic} if any projection in it has an atomic
subprojection. It is easy to see that this happens if and only if $\A = (\sum_i B(H_i))_\infty$.
The ``if'' direction is easy.
Conversely, if $\A$ is purely atomic,
denote by $(e_i)_{i \in I}$ a maximal collection of mutually non-equivalent atomic
projections in $\A$. Denote by $z(p)$ the central cover of $p$.
Then $z(e_i) z(e_j) = 0$ if $i \neq j$, and $\sum_i z(e_i) = 1$.
Consequently, $\A = \sum_i z(e_i) \A$. For a fixed $i$, let $(f_j)_{j \in J(i)}$
be a maximal family of mutually orthogonal atomic projections, so that $e_i$
is one of these projections. The $f_j$'s have the same central cover (namely, $z(e_i)$),
hence they are all equivalent to $e_i$. Furthermore, $z(e_i) = \sum_{j \in J(i)} f_j$,
hence $z(e_i) \A$ is isomorphic to $B(\ell_2(J(i)))$.

\begin{proof}
If a von Neumann algebra $\A$ is not purely atomic, then, as explained in
\cite[Section 1]{OS}, $\A_*$ contains a (complemented) copy of $L_1(0,1)$.
This establishes the ``only if'' implication of Proposition \ref{p:predual}.
Conversely, if $\A$ is purely atomic, then $\A_*$ is isometric to a (contractively
complemented) subspace of $\cs_1(H)$, and the latter is subprojective.
\end{proof}

\section{$p$-convex and $p$-disjointly homogeneous Banach lattices}\label{s:DH}

We say that $X$ is \emph{$p$-disjointly homogeneous} (\emph{$p$-DH} for short)
if every disjoint normalized sequence contains a subsequence equivalent
to the standard basis of $\ell_p$.

For the sake of completeness we present a proof of the following statement
(see \cite[4.11, 4.12]{FHSTT}).

\begin{proposition} \label{p-DC}
 Let $X$ be a $p$-convex. Then every subspace, spanned by a disjoint sequence equivalent
to the canonical basis of $\ell_p$, is complemented.
\end{proposition}

\begin{proof}
Let $(x_k) \subset X$ be a disjoint normalized sequence. Since $X$ is DH, by passing to a subsequence,
($x_k)$ is an  $\ell_p$ basic sequence. Then,  in the  p-concavification $X_{(p)}$ the disjoint
sequence $({x_k}^p)$ is an $\ell_1$ basic sequence. Therefore, there exists a functional
$x^* \in [({x_k}^p)]$ such that $x^*({x_k}^p)=1$ for all $ k$. By the Hahn-Banach Theorem $x^*$
can be extended to a positive functional in ${X_{(p)}}^*$. Define a seminorm
$\|x\|_p=(x^*(|x^p|))^{\frac{1}{p}}$ on $X$. Denote by $\mathcal{N}$ the subset of $X$
on which this seminorm is equal to zero. Clearly, $\mathcal{N}$ is an ideal, therefore,
the quotient space $\tilde{X}=X/\mathcal{N}$ is a Banach lattice, and the quotient map
$Q: X\to \tilde{X}$ is an orthomorphism. With the defined seminorm  $\tilde{X}$
is an abstract $L_p$-space, and  the disjoint sequence $Q(x_k)$ is normalized.
Therefore it is an $\ell_p$ basic sequence that spans a complemented subspace
(in particular, $Q$ is an isomorphism when restricted to $[x_k]$).
Let $\tilde{P}$ be a projection from $\tilde{X}$ onto $[Q(x_k)]$.
Then  $P=Q^{-1}\tilde{P}Q$ is a projection from $X$ onto $[x_k]$.
\end{proof}


\begin{proposition}\label{p:p-DH l2 lp}
Let $X$ be a $p$-convex,  $p$-disjointly homogeneous Banach lattice $( p \ge 2)$.
Then any subspace of $X$ contains
a complemented copy of either $\ell_p$ or $\ell_2$. Consequently, $X$ is subprojective.
\end{proposition}

\begin{proof} First,  note that $X$ is order continuous.
 Let $M \subseteq X$ be an infinite dimensional separable subspace. Then there exists
a complemented order ideal  in $X$ with a weak unit that contains  $M$. Therefore,
without loss of generality, we may assume that $X$ has a weak unit. Then there exists
a probability measure $\mu$  \cite[p. 14]{JMST} such that we have continuous embeddings
$$L_{\infty}(\mu) \subseteq X \subseteq L_p(\mu) \subseteq L_2(\mu) \subseteq L_1(\mu).$$
Consequently, there exists a constant $c_1 > 0$ so that $c_1 \|x\|_p \leq \|x\|$
for any $x \in X$.

By the proof of \cite[Proposition 1.c.8]{LT2}, one of the following holds:

\begin{case}
 $M$ contains an almost disjoint bounded sequence.  By Proposition~\ref{p-DC} $M$
contains a copy of $\ell_p$ complemented in $X$.
\end{case}

\begin{case}
The norms $\| \cdot \|$ and $\| \cdot \|_1$ are equivalent on $M$.
Thus, there exists $c_2 > 0$ so that, for any $y \in M$,
$$
c_2 \|y\|_2 \geq c_2 \|y\|_1 \geq \|y\| \geq c_1 \|y\|_p \geq c_1 \|y\|_2 .
$$
In particular, $M$ is embedded into $L_2(\mu)$ as a closed subspace.
The orthogonal projection from $L_2(\mu)$ onto $M$ then defines a bounded
projection from $X$ onto $M$.
\qedhere
\end{case}
\end{proof}

The preceding result implies that Lorentz space $ \Lambda_{p,W}(0,1)$ is subprojective  since it is  $p$-DH and $p$-convex $(p \ge 1)$, see \cite[Theorem 3]{FJT} and \cite{KMP}.  Note that, originally,  the subprojectivity of $\Lambda(p,W)$ $(p \ge 2)$ was  observed in \cite[Remark~5.7]{FJT}. 

\section{Lattice-valued $\ell_p$ spaces}\label{s:lattice}

If $X$ is a Banach lattice, and $1 \leq p < \infty$, denote by $\widetilde{X(\ell_p)}$
the completion of the space of all finite sequences $(x_1, \ldots, x_n)$ (with $x_i \in X$),
equipped with the norm $\|(x_1, \ldots, x_n)\| = \|(\sum_i |x_i|^p)^{1/p}\|$, where
$$
(\sum_i |x_i|^p)^{1/p} =
\sup \Big\{ |\sum_i \alpha_i x_i| : \sum_i |\alpha_i|^{p^\prime} \leq 1 \Big\} ,
{\textrm{     with    }} \frac{1}{p} + \frac{1}{p^\prime} = 1 .
$$
See \cite[pp.~46-48]{LT2} for more information.
We have:

\begin{proposition}\label{p:X(l_p)}
Suppose $X$ is a subprojective separable space, with the lattice structure
given by an unconditional basis, and $1 \leq p < \infty$. Then
$\widetilde{X(\ell_p)}$ is subprojective.
\end{proposition}

\begin{proof}
To show that any subspace $Y \subset \widetilde{X(\ell_p)}$ has a further subspace $Z$,
complemented in $\widetilde{X(\ell_p)}$,
let $x_1, x_2, \ldots$ and $e_1, e_2, \ldots$ be the canonical 
bases in $X$ and $\ell_p$,
respectively. Then the elements $u_{ij} = x_i \otimes e_j$ form an unconditional basis
in $\widetilde{X(\ell_p)}$, with
\begin{equation}
\|\sum a_{ij} u_{ij}\| = \|\sum_i (\sum_j |a_{ij}|^p)^{1/p} x_i\|_X =
\|\sum_i \big( \sup_{\sum_j |\alpha_j|^{p^\prime} \leq 1} |\sum_j \alpha_{ij} a_{ij}| \big) x_i\|_X .
\label{eq:X(l_p)}
\end{equation}
Let $P_n$ be the canonical projection onto $\span[u_{ij} : 0 \leq i \leq n, j \in \N]$,
and set $P_n^\perp = I - P_n$. The range of $P_n$ is isomorphic to $\ell_p$, hence,
if $P_n|_Y$ is not strictly singular for some $n$, we are done, by Corollary \ref{c:complem}.
If $P_n|_Y$ is strictly singular
for every $n$, find a normalized sequence $(y_i)$ in $Y$, and $1 = n_1 < n_2 < \ldots$, so that
$\|P_{n_i} y_i\|, \|P_{n_{i+1}}^\perp y_i\| < 100^{-i}/2$. By small perturbation, it remains
to prove the following: if $y_i = P_{n_i}^\perp P_{n_{i+1}} y_i$, then $\span[y_i : i \in \N]$
contains a subspace, complemented in $\widetilde{X(\ell_p)}$. Further, we may assume that
for each $i$ there exists $M_i$ so that
we can write
$$
y_i = \sum_{n_i < k \leq n_{i+1}, 1 \leq j \leq M_i} a_{kj} u_{kj} .
$$
For each $k \in [n_i+1, n_{i+1}]$ (and arbitrary $i \in \N$) find a finite sequence
$(\alpha_{kj})_{j=1}^{M_i}$ so that $\sum_j |\alpha_{kj}|^{p^\prime} = 1$, and
$|\sum_j \alpha_{kj} a_{kj}| = (\sum_j |a_{kj}|^p)^{1/p}$.
Define $U : \widetilde{X(\ell_p)} \to X : u_{kj} \mapsto \alpha_{kj} a_{kj} x_k$.
By \eqref{eq:X(l_p)}, $U$ is a contraction, and $U|_{\span[y_i : i \in \N]}$ is an isometry.
To finish the proof, recall that $X$ is subprojective, and apply Corollary \ref{c:complem}.
\end{proof}

\begin{remark}\label{r:C(K) l_p}
Using similar methods, one can prove: if $K$ is a compact metrizable space,
and $1 \leq p < \infty$, then $\widetilde{C(K)(\ell_p)}$ is subprojective.
\end{remark}

Recall that, for a Banach space $X$, we denote by $Rad(X)$ the completion of the finite sums
$\sum_n r_n x_n$ ($r_1, r_2, \ldots$ are Rademacher functions, and $x_1, x_2, \ldots \in X$)
in the norm of $L_1(X)$ (equivalently, by Khintchine-Kahane Inequality, in the norm of $L_p(X)$).
If $X$ has a unconditional basis $(x_i)$ and finite cotype,
then $Rad(X)$ is isomorphic to $\widetilde{X(\ell_2)}$
(here we can view $X$ as a Banach lattice, with the order induced by the
basis $(x_i)$). Indeed, by \cite[Section 1.f]{LT2}, $X$ is $q$-concave, for some $q$.
An array $(a_{mn})$ can be identified both with an element of $Rad(X)$
(with the norm $\int_0^1 \| \sum_m \sum_n a_{mn}r_nx_m\|$), and with an element of
$\widetilde{X(\ell_2)}$ (with the norm $\| \sum_m (\sum_n |a_{mn}|^2)^{1/2}x_m\|$).
Then
$$
\eqalign{
& D \| \sum_m (\sum_n |a_{mn}|^2)^{1/2}x_m\|
  \le 
 \|  \sum_m \int_0^1 |\sum_n a_{mn}r_n|x_m\| 
 =
 \| \int_0^1 | \sum_m \sum_n a_{mn}r_nx_m| \| \cr
&\le
 \int_0^1 \| \sum_m \sum_n a_{mn}r_nx_m\|  
 \le
 (\int_0^1 \|\sum_m \sum_n a_{mn}r_nx_m\|^q)^{1/q}\cr
  &\le  
M_q \|(\int_0^1 |\sum_m \sum_n a_{mn}r_nx_m|^q)^{1/q}\|
 \le M_q \|\sum_m (\int_0^1|\sum_n a_{mn}r_n|^q)^{1/q}x_m\| \cr
&\le
C M_q\| \sum_m (\sum_n |a_{mn}|^2)^{1/2}x_m\|,
}
$$
where $M_q$ is a $q$-concavity constant, while $D$ and $C$ come from Khintchine's inequality.
Thus, we have proved:

\begin{proposition}\label{p:Rad X}
If $X$ is a subprojective space with an unconditional basis and non-trivial cotype,
then $Rad(X)$ is subprojective.
\end{proposition}

\begin{remark}
%
By \cite[Theorem~2.3]{KW}, if $X$ is a non-atomic order continuous 
Banach lattice with an unconditional basis,
then $\widetilde{X(\ell_2)}$ is isomorphic to $X$.
Furthermore, if $X$ is a non-atomic Banach lattice with an unconditional basis and
non-trivial cotype, then $Rad(X)$ is isomorphic to $X$.
Indeed, non-trivial cotype implies non-trivial lower estimate \cite[p.~100]{LT2},
which, by \cite[Theorem 2.4.2]{M-N}, implies order continuity. Therefore,
$X$ is isomorphic to $\widetilde{X(\ell_2)}$, which, in turn, is isomorphic to $Rad(X)$.
\end{remark}

\end{document}